\documentclass[a4paper,10pt]{amsart}
\usepackage[english]{babel}
\usepackage{amsmath,tikz-cd}
\usepackage{amssymb}
\usepackage[T1]{fontenc}    
\usepackage[utf8]{inputenc} 
\usepackage{lmodern}
\usepackage{mathrsfs}
\usepackage{enumerate}
\usepackage{enumitem}
\usepackage{mathtools,yfonts}
\usepackage{tikz,tkz-euclide}
\usepackage{todonotes}
\usepackage{pgfplots}
\pgfplotsset{compat=newest}
\usetikzlibrary{arrows}
\usepackage{makecell}

\usepackage{xcolor} 
\definecolor{DarkRed}{RGB}{173,0,0}
\definecolor{LightRed}{RGB}{201,0,0}
\usepackage[
    colorlinks=true,
    linkcolor=DarkRed,
    urlcolor=LightRed,
    citecolor=LightRed
]{hyperref}

\usepackage{tcolorbox}
\usepackage[]{algorithm2e}

\newtheorem{thm}{Theorem}[section]

\newtheorem{Proposition}[thm]{Proposition}

\theoremstyle{definition}

\newtheorem{Definition}[thm]{Definition}
\newtheorem{Remark}[thm]{Remark}

\newtheorem{Example}[thm]{Example}

\newtheoremstyle{introthmstyle} {5pt}{5pt} {\itshape} {} {\bfseries} {.} { } {} \theoremstyle{introthmstyle} \newtheorem{introthm}{Theorem}

\definecolor{wwwwww}{rgb}{0.4,0.4,0.4}

\newcommand{\PP}{\mathbb{P}}
\newcommand{\ZZ}{\mathbb{Z}}
\newcommand{\QQ}{\mathbb{Q}}
\newcommand{\RR}{\mathbb{R}}
\newcommand{\CC}{\mathbb{C}}

\newcommand{\OO}{\mathcal{O}}   

\DeclareMathOperator{\Cox}{Cox} 
\DeclareMathOperator{\Spec}{Spec}
\DeclareMathOperator{\Pic}{Pic}

\DeclareMathOperator{\Aut}{Aut}

\DeclareMathOperator{\Mult}{Mult}

\newcommand{\cM}{\overline{\mathcal{M}}}
\DeclareMathOperator{\Bl}{Bl}
\DeclareMathOperator{\Cl}{Cl}
\DeclareMathOperator{\NE}{NE}

\DeclareMathOperator{\Eff}{Eff}

\DeclareMathOperator{\Exc}{Exc}

\DeclareMathOperator{\PsAut}{PsAut}

\DeclareMathOperator{\Par}{Par}
\newcommand{\hkquot}{/\!\!/\!\!/}
\DeclareMathOperator{\Cone}{Cone}
\DeclareMathOperator{\Nef}{Nef}
\DeclareMathOperator{\Mov}{Mov}

\newcommand{\quot}{/\hspace{-1.2mm}/}

\hypersetup{pdfpagemode=UseNone}
\hypersetup{pdfstartview=FitH}

\DeclareMathOperator{\PGL}{PGL}

\DeclareMathOperator{\GL}{GL}

\newcommand{\bb}{\mathbf b}
\newcommand{\avec}{\mathbf a}

\newcommand{\Padd}{\mathcal P^{\rm add}}
\newcommand{\Mtriv}{M^{\rm triv}}

\setlist[enumerate]{
  label=(\roman*),
  before=\normalfont,
  font=\normalfont
}

\begin{document}

\title{Mori theory and symplectic reduction for polygons and quivers}

\author[Alessia Mandini]{Alessia Mandini}
\address{\sc Alessia Mandini\\ Dipartimento di Informatica, Università di Verona, Strada le Grazie 15, 37134 Verona, Italy}
\email{alessia.mandini@univr.it}

\author[Alex Massarenti]{Alex Massarenti}
\address{\sc Alex Massarenti\\ Dipartimento di Matematica e Informatica, Università di Ferrara, Via Machiavelli 30, 44121 Ferrara, Italy}
\email{msslxa@unife.it}

\date{\today}

\begin{abstract}
We investigate the birational geometry of moduli spaces of polygons and quivers, emphasizing the link between Mori theory and symplectic reduction. On the additive side we relate Mori models of $\Bl_{n-1}\PP^{n-3}$ with GIT quotients of configurations of points on $\PP^1$, Hassett--GIT spaces, parabolic quotients with fixed trivial bundle, quiver moduli spaces, and additive polygon spaces. On the multiplicative side we relate Mori models of $\Bl_n\PP^{n-3}$ with moduli spaces of rank two parabolic bundles with trivial determinant, multiplicative polygons, multiplicative configurations, and compact multiplicative quiver quotients. In this way, the Mori chamber decompositions are matched with wall-crossing for the corresponding additive and multiplicative symplectic quotients. Finally, for the small models appearing in these two constructions, we compute the corresponding groups of algebraic automorphisms.
\end{abstract}

\maketitle
\setcounter{tocdepth}{1}
\tableofcontents

\section*{Introduction}
Moduli spaces frequently admit several different constructions. The same projective variety may arise as a GIT quotient, as a moduli space of parabolic bundles, as a quiver moduli space, as a symplectic quotient, or as a birational model of a blow-up of a projective space. One of the main advantages of Mori dream spaces is that their birational geometry is controlled by a finite amount of data. Indeed, the cone of effective divisors of a Mori dream space admits a decomposition into convex sets, called Mori chambers, and the chambers are the nef cones of birational models of the variety. These varieties were introduced by Hu and Keel in \cite{HuKeel2000}, and they behave in the best possible way from the point of view of Mori's minimal model program.

The aim of this paper is to study parallel, moduli-theoretic pictures from the point of view of Mori chambers. The first one is additive and is governed by the blow-up $Y_{n-3}:=\Bl_{q_1,\ldots,q_{n-1}}\PP^{n-3}$ of $\PP^{n-3}$ at $n-1$ general points. The second one is multiplicative and is governed by the blow-up $Z_n:=\Bl_{q_1,\ldots,q_n}\PP^{n-3}$ of $\PP^{n-3}$ at $n$ general points. These blow-ups lie at the boundary of the classical finite generation results for Cox rings of blow-ups of projective spaces at general points, and their Mori chamber decompositions are rich enough to encode non-trivial wall-crossing phenomena \cite{Mukai2005,CastravetTevelev2006,AraujoMassarenti2016}.

We first consider the additive side. Let $\bb=(b_1,\ldots,b_n)\in\QQ_{>0}^n$. The GIT quotient $((\PP^1)^n)^{ss}(\bb)\quot\PGL_2$ parametrizes weighted configurations of ordered points on $\PP^1$. Such quotients are classical objects in invariant theory and are closely related to the Gelfand--MacPherson correspondence, to linear systems on blow-ups of projective spaces, and to the birational geometry of compactifications of $M_{0,n}$ \cite{GelfandMacPherson1982,Kumar2003,BolognesiMassarenti2021}. After normalizing the weights by $a_i=2b_i/|\bb|$, so that $\sum_i a_i=2$, the same quotient gives the corresponding Hassett--GIT space $\cM_{0,\avec}^{\mathrm{GIT}}$; we refer to \cite{Kapranov1993,Hassett2003,BolognesiMassarenti2021} for these spaces and their birational models.

The quotient $((\PP^1)^n)^{ss}(\bb)\quot\PGL_2$ also has two further, and very useful, interpretations. It is the moduli space of parabolic structures on the fixed trivial bundle $\OO_{\PP^1}^{\oplus 2}$ with the same linearization, and it is the moduli space of semistable representations of the star-shaped quiver with dimension vector $(2,1,\ldots,1)$. Finally, by the Kempf--Ness theorem, this quotient is identified with the additive polygon space parametrizing closed Euclidean polygons with side lengths $b_1,\ldots,b_n$, modulo rotations. Polygon spaces have been widely studied from the symplectic, topological and algebraic viewpoints \cite{Klyachko1994,KapovichMillson1996,HausmannKnutson1997,HausmannKnutson1998,Mandini2014,MandiniPabiniak2017}.

The bridge with Mori theory is given by the divisor class
$D_{\bb,n}$ on $Y_{n-3}$. The Mori model of $Y_{n-3}$ associated with the chamber containing $D_{\bb,n}$ is the common projective variety underlying the additive moduli spaces above. We summarize the main results on the additive side in the following statement. The chain of identifications is proved in Theorem \ref{thm:additive-summary}, while the description of the algebraic automorphism groups of the small models follows from Theorem \ref{thm:additive-small-aut}.

\begin{introthm}\label{introthm:additive}
Let $\bb=(b_1,\ldots,b_n)\in\QQ_{>0}^n$ be such that the semistable locus of $(\PP^1)^n$ with respect to the linearization
$L_{\bb}=\OO_{\PP^1}(b_1)\boxtimes\cdots\boxtimes\OO_{\PP^1}(b_n)$ is non-empty. Set $|\bb|=\sum_i b_i$, $\bar b=|\bb|/2$, and
$\avec=(2b_1/|\bb|,\ldots,2b_n/|\bb|)$. Fix the label $n$, let
$Y_{n-3}=\Bl_{q_1,\ldots,q_{n-1}}\PP^{n-3}$, and let $Y_{\bb,n}$ be the Mori model of $Y_{n-3}$ associated with the divisor class $D_{\bb,n}=(\bar b-b_n)H-\sum_{i=1}^{n-1}(\bar b-b_n-b_i)E_i$. Then there are natural identifications
$$
Y_{\bb,n}\cong (\PP^1)^n_{\bb}\cong \cM_{0,\avec}^{\mathrm{GIT}}
\cong \Mtriv_{\bb}\cong \mathcal M_Q(\bb)\cong \Padd_{\bb}.
$$
Here $(\PP^1)^n_{\bb}$ is the GIT quotient of weighted configurations of $n$ points on $\PP^1$, $\cM_{0,\avec}^{\mathrm{GIT}}$ is the corresponding Hassett--GIT model, $\Mtriv_{\bb}$ is the moduli space of parabolic structures on the fixed trivial bundle $\OO_{\PP^1}^{\oplus 2}$, $\mathcal M_Q(\bb)$ is the moduli space of semistable representations of the star-shaped quiver with dimension vector $(2,1,\ldots,1)$ and stability parameter $\theta_{\bb}=(-|\bb|,2b_1,\ldots,2b_n)$, and $\Padd_{\bb}$ is the additive polygon space with side lengths $b_1,\ldots,b_n$. More precisely, the first five spaces are isomorphic as projective varieties, while the last identification is the Kempf--Ness identification with the corresponding symplectic quotient.

Assume moreover that the chamber containing $D_{\bb,n}$ is contained in $\Mov(Y_{n-3})$, so that $Y_{\bb,n}$ is a small $\QQ$-factorial modification of $Y_{n-3}$. Let
$
\mathcal K_{\bb}
=
\{I\subset\{1,\ldots,n\}\mid \sum_{i\in I}a_i<1\}
$
be the collision complex of the corresponding chamber, and set
$
\Gamma_{\bb}:=\Aut(\mathcal K_{\bb}).
$
Then the algebraic automorphism groups of the projective varieties in the above chain are naturally identified and
$$
\Aut(Y_{\bb,n})
\cong
\Aut((\PP^1)^n_{\bb})
\cong
\Aut(\cM_{0,\avec}^{\mathrm{GIT}})
\cong
\Aut(\Mtriv_{\bb})
\cong
\Aut(\mathcal M_Q(\bb))
\cong
\Aut(\Padd_{\bb})
\cong
\Gamma_{\bb},
$$
where $\Aut(\Padd_{\bb})$ refers to the projective structure induced by the GIT/Kempf--Ness identification. In modular terms, the group $\Gamma_{\bb}$ is the group of relabellings preserving the collision combinatorics of the chamber; it acts on configurations and Hassett markings by relabelling, on $\Mtriv_{\bb}$ by relabelling the parabolic directions, on $\mathcal M_Q(\bb)$ by permuting the legs of the quiver, and on $\Padd_{\bb}$ by permuting the corresponding sides, followed, when necessary, by the canonical identification inside the same chamber.
\end{introthm}

The first five spaces are naturally isomorphic as projective varieties, while the last identification is the Kempf--Ness identification with the corresponding symplectic quotient. We also describe explicitly the chamber giving the original blow-up $Y_{n-3}$, and, in the case $Y_4=\Bl_6\PP^4$, the sequence of flips leading from the blow-up chamber to the Fano model. In quiver terms these flips are the wall-crossings obtained by varying the stability parameter of the star-shaped quiver.

On the multiplicative side, the additive closing equation is replaced by a group-valued equation. For a weight $A=(a_1,\ldots,a_n)$ with $0<a_i<1$, let $\mathcal C_i\subset SU(2)$ be the conjugacy class with eigenvalues $\exp(\pi i a_i)$ and $\exp(-\pi i a_i)$. The multiplicative polygon space is the compact unitary character variety of the punctured sphere with prescribed local monodromy. Such character varieties and their symplectic structures play a central role in the study of representation spaces and moduli spaces of flat connections \cite{Goldman1984,GuruprasadHuebschmannJeffreyWeinstein1997,AlekseevMalkinMeinrenken1998}. By the Mehta--Seshadri correspondence \cite{MehtaSeshadri1980}, this character variety is identified with the moduli space $M_A$ of rank two parabolic bundles on $\PP^1$ with trivial determinant and weights $A$. In the rank two genus zero case these spaces have been studied in relation with representation spaces, elliptic surfaces, blow-ups, and automorphism groups \cite{Bauer1991,BiswasHollaKumar2010,Casagrande2015,AraujoFassarellaKaurMassarenti2019}.

The same multiplicative data have two further descriptions. First, each conjugacy class $\mathcal C_i$ is identified with $\PP^1$ by sending an element to its eigenline. This yields a space of multiplicative configurations, which is a configuration-theoretic incarnation of the unitary character variety. Secondly, the same equation can be read as a compact multiplicative quiver quotient. This compact picture is related to the algebraic theory of multiplicative quiver varieties and multiplicative projective algebras, studied in connection with middle convolution, the Deligne--Simpson problem, and character varieties \cite{CrawleyBoeveyShaw2006,Yamakawa2008}.

The Mori-theoretic model on the multiplicative side is controlled by the Bauer--Mukai chamber decomposition of the demi-hypercube. Mukai's description relates these chambers to parabolic bundles and to finite generation of Nagata invariant rings \cite{Mukai2005}. If $A$ belongs to a chamber of this decomposition, the corresponding Mori model $Z_A$ of $Z_n$ is naturally isomorphic to $M_A$. We summarize the main results on the multiplicative side in the following statement. The chain of identifications is proved in Theorem \ref{thm:multiplicative-summary}, while the statements on algebraic automorphisms follow from Theorems \ref{thm:aut-central-multiplicative-polygons} and \ref{thm:aut-general-multiplicative-polygons}.

\begin{introthm}\label{introthm:multiplicative}
Let $A=(a_1,\ldots,a_n)\in\Delta$ be a generic weight, with $0<a_i<1$ for all $i$. Let
$Z_n=\Bl_{q_1,\ldots,q_n}\PP^{n-3}$, let $C_A$ be the chamber of the Bauer--Mukai decomposition of $\Delta$ containing $A$, and let $Z_A$ be the corresponding Mori model of $Z_n$. Then there are natural identifications
$$
Z_A\cong M_A\cong \Mult_A\cong
\operatorname{Conf}^{\operatorname{mult}}_A(\PP^1)/SU(2)
=
\mathcal M_Q^{\operatorname{mult},c}(A).
$$
Here $M_A$ is the moduli space of rank two parabolic bundles on $\PP^1$ with trivial determinant and parabolic weights $A$, $\Mult_A$ is the multiplicative polygon space, or compact unitary character variety, with prescribed conjugacy classes determined by the eigenvalues $\exp(\pm \pi i a_i)$, $\operatorname{Conf}^{\operatorname{mult}}_A(\PP^1)/SU(2)$ is the corresponding eigenline description of the same quotient, and $\mathcal M_Q^{\operatorname{mult},c}(A)$ is the compact multiplicative quiver quotient. More precisely, $Z_A\cong M_A$ is an isomorphism of projective varieties, while $M_A\cong\Mult_A$ is the Mehta--Seshadri correspondence: it is a homeomorphism and, on the stable locus, a real analytic symplectomorphism. The remaining identifications are induced by the eigenline description of the conjugacy classes and by the compact multiplicative quiver presentation.

Let $El\cong(\ZZ/2\ZZ)^{n-1}$ be the group of elementary transformations, indexed by even subsets of $\{1,\ldots,n\}$, and let $El_A\subset El$ be the subgroup preserving the Bauer--Mukai chamber of $A$. For the weights $A$ in the range of Theorem \ref{thm:aut-general-multiplicative-polygons}, the algebraic automorphism groups of the projective varieties in the above chain are naturally identified and
$$
\Aut(Z_A)
\cong
\Aut(M_A)
\cong
\Aut(\Mult_A)
\cong
\Aut(\operatorname{Conf}^{\operatorname{mult}}_A(\PP^1)/SU(2))
\cong
\Aut(\mathcal M_Q^{\operatorname{mult},c}(A))
\cong
El_A,
$$
where the automorphism groups of $\Mult_A$, $\operatorname{Conf}^{\operatorname{mult}}_A(\PP^1)/SU(2)$ and $\mathcal M_Q^{\operatorname{mult},c}(A)$ are taken with respect to the projective structures transported from $M_A$. In the central case $A_F=(1/2,\ldots,1/2)$ one has $El_{A_F}=El$, and hence the common algebraic automorphism group is $(\ZZ/2\ZZ)^{n-1}$. In modular terms, an elementary transformation changes the parabolic bundle by an elementary modification at an even set of marked points; in the multiplicative polygon, eigenline and quiver descriptions it induces the corresponding operation on the prescribed unitary monodromy data.
\end{introthm}

Here $Z_A\cong M_A$ is an isomorphism of projective varieties, while $M_A\cong\Mult_A$ is the Mehta--Seshadri correspondence; in particular, on the stable locus it is a real analytic symplectomorphism.

We stress that the additive and the multiplicative pictures are parallel but not identical. In the additive case the projective algebraic structure is built in from the beginning, since the polygon space is identified with a GIT quotient. In the multiplicative case the quotient $\Mult_A$ is naturally a compact symplectic quotient, and the complex projective structure used in this paper is the one transported from the parabolic moduli space $M_A$. This distinction is important in the study of automorphisms below.

The following table summarizes the two dictionaries developed in the paper. The last column records the nature of the identifications and the place where they are proved.

\begin{center}
\renewcommand{\arraystretch}{1.35}
\tiny
\begin{tabular}{p{0.30\textwidth}|p{0.30\textwidth}|p{0.30\textwidth}}
\textbf{Additive side}
&
\textbf{Multiplicative side}
&
\textbf{Nature of the identification}
\\
\hline
Model $Y_{\bb,n}$ of
$
Y_{n-3}=\Bl_{q_1,\ldots,q_{n-1}}\PP^{n-3}
$
&
Model $Z_A$ of
$
Z_n=\Bl_{q_1,\ldots,q_n}\PP^{n-3}
$
&
Birational Mori models. Theorems \ref{thm:additive-summary} and \ref{thm:multiplicative-summary}.
\\
\hline
Weighted configuration quotient
$(\PP^1)^n_{\bb}$
and Hassett--GIT model
$\cM_{0,\avec}^{\mathrm{GIT}}$
&
Multiplicative configuration quotient
$\operatorname{Conf}^{\operatorname{mult}}_A(\PP^1)/SU(2)$
&
On the additive side the identification is algebraic and follows from the same collision rule. On the multiplicative side the identification is the eigenline description of unitary conjugacy classes. See Proposition \ref{prop:git-hassett-fixed} and Theorem \ref{thm:multiplicative-summary}.
\\
\hline
Fixed-trivial-bundle parabolic quotient
$\Mtriv_{\bb}$
&
Trivial-determinant parabolic bundle moduli space
$M_A$
&
The additive space parametrizes parabolic structures on the fixed bundle $\OO_{\PP^1}^{\oplus 2}$. The multiplicative space parametrizes rank two parabolic bundles with trivial determinant and variable underlying bundle. See Proposition \ref{prop:fixed-trivial-git} and Theorem \ref{thm:BM-model}.
\\
\hline
Ordinary additive quiver moduli space
$\mathcal M_Q(\bb)$
&
Compact multiplicative quiver quotient
$\mathcal M_Q^{\operatorname{mult},c}(A)$
&
The additive identification is algebraic, since the outer scalar actions projectivize the vectors. The multiplicative equality is a quotient-theoretic identity with $\Mult_A$. See Section \ref{Sec2} and Proposition \ref{prop:mult-quiver-polygon}.
\\
\hline
Additive polygon space
$\Padd_{\bb}$
&
Multiplicative polygon space
$\Mult_A$
&
The additive identification is the Kempf--Ness identification with the GIT quotient. The multiplicative identification is the Mehta--Seshadri correspondence, a homeomorphism and, on the stable locus, a real analytic symplectomorphism.
\end{tabular}
\end{center}

We also compute automorphism groups. In the multiplicative case we will use the description of automorphisms of moduli spaces of rank two parabolic bundles with trivial determinant. Elementary transformations indexed by even subsets of the markings generate a group isomorphic to $(\ZZ/2\ZZ)^{n-1}$. For the central weight this is the full algebraic automorphism group, and for the small modifications considered here the group is the subgroup of elementary transformations preserving the chamber. Via the multiplicative dictionary, the same statement gives the algebraic automorphism groups of the corresponding multiplicative polygons, multiplicative configurations, and compact multiplicative quiver quotients.

In the additive case we prove the analogous statement for small $\QQ$-factorial modifications of $Y_{n-3}$. If $\mathcal C$ is a Mori chamber, and $Y_{\mathcal C}$ is the corresponding small model, then $\Aut(Y_{\mathcal C})$ is the group of automorphisms of the collision complex attached to the chamber. Therefore the automorphisms are precisely the relabellings preserving the corresponding collision combinatorics. Through the additive dictionary, the same group is the group of algebraic automorphisms of the corresponding configuration quotients, Hassett--GIT spaces, parabolic quotients with fixed trivial bundle, quiver moduli spaces, and additive polygon spaces. We emphasize that these are algebraic, or projective, automorphism groups. They should not be confused with the full symplectomorphism group of the underlying symplectic manifolds, which is in general much larger.

\subsection*{Organization of the paper}
All through the paper we work over an algebraically closed field of characteristic zero, unless otherwise stated. In Section \ref{Sec1} we recall the basic facts on Mori dream spaces and the Mori chamber decompositions of the blow-ups of projective spaces that will be used later. In Section \ref{Sec2} we develop the additive dictionary, passing from configurations of points on $\PP^1$ to parabolic quotients with fixed trivial bundle, quiver moduli spaces, and additive polygon spaces. In Section \ref{Sec3} we develop the multiplicative dictionary, relating parabolic bundles with trivial determinant, multiplicative polygons, multiplicative configurations, and compact multiplicative quiver quotients. Finally, in Section \ref{Sec4} we compute the automorphism groups of the small models appearing in the two constructions.

\subsection*{Acknowledgments}
The authors are members of GNSAGA (INdAM). A.~Massarenti was supported by the PRIN 2022 project 20223B5S8L, “Birational Geometry of Moduli Spaces and Special Varieties”. We thank Pieter Belmans for his valuable comments on an earlier version of this preprint.

\section{Mori dream spaces and blow-ups of projective spaces}\label{Sec1}

Let $X$ be a normal projective variety over an algebraically closed field of characteristic zero. We will denote by $N^1(X)$ the real vector space of $\RR$-Cartier divisors modulo numerical equivalence, and by $N_1(X)$ the dual vector space of numerical classes of $1$-cycles. The nef cone, movable cone, and effective cone are denoted respectively by $\Nef(X)$, $\Mov(X)$, and $\Eff(X)$. Thus $\Nef(X)$ is the closed cone generated by nef divisor classes, $\Mov(X)$ is the cone generated by movable divisor classes, and $\Eff(X)$ is the cone generated by effective divisor classes. We have inclusions
$$
\Nef(X)\subset \Mov(X)\subset \Eff(X).
$$
The Mori cone $\NE(X)\subset N_1(X)$ is the closed cone generated by classes of effective curves.

A birational map $f:X\dashrightarrow Y$ to a normal projective variety is a birational contraction if $f^{-1}$ does not contract any divisor. It is a small $\QQ$-factorial modification if $Y$ is $\QQ$-factorial and $f$ is an isomorphism in codimension one. If $f:X\dashrightarrow Y$ is a small $\QQ$-factorial modification, then the pull-back identifies $N^1(Y)$ with $N^1(X)$ and sends $\Eff(Y)$ and $\Mov(Y)$ isomorphically onto $\Eff(X)$ and $\Mov(X)$.

\begin{Definition}\label{def:MDS}
A normal projective $\QQ$-factorial variety $X$ is a Mori dream space if the following conditions hold:
\begin{enumerate}
\item $\Pic(X)$ is finitely generated;
\item $\Nef(X)$ is generated by finitely many semiample divisor classes;
\item there is a finite collection of small $\QQ$-factorial modifications $f_i:X\dashrightarrow X_i$ such that each $X_i$ satisfies the second condition and
$$
\Mov(X)=\bigcup_i f_i^*\Nef(X_i).
$$
\end{enumerate}
\end{Definition}

For a Mori dream space, the fan on $\Mov(X)$ obtained from the nef cones of its small $\QQ$-factorial modifications extends to a fan on the whole effective cone.

\begin{Definition}\label{def:MCD}
Let $X$ be a Mori dream space. The Mori chamber decomposition of $\Eff(X)$ is the fan whose maximal cones are obtained as follows. Let $g_i:X\dashrightarrow Y_i$ run through the finitely many birational contractions from $X$ to Mori dream spaces. If $\Exc(g_i)$ is the set of exceptional prime divisors of $g_i$, then the corresponding maximal chamber is
$$
\Cone\left(g_i^*\Nef(Y_i),\Exc(g_i)\right).
$$
\end{Definition}

The Cox ring viewpoint gives the bridge with variation of GIT. If $\Cl(X)$ is finitely generated, the Cox ring is
$$
\Cox(X)=\bigoplus_{[D]\in \Cl(X)} H^0(X,\OO_X(D)),
$$
with its natural $\Cl(X)$-grading. When $\Cox(X)$ is finitely generated, $X$ is obtained as a GIT quotient of an open subset of $\Spec\Cox(X)$ by the Picard quasi-torus $\Spec \CC[\Cl(X)]$. In this description, the GIT chambers for this torus action are the Mori chambers of $X$. In general the stable base locus decomposition of $\Eff(X)$ is coarser than the Mori chamber decomposition.

We now specialize to blow-ups of projective spaces. Let $q_1,\ldots,q_k\in \PP^r$ be points in general position, and set
$$
X^r_k:=\Bl_{q_1,\ldots,q_k}\PP^r.
$$
We will denote by $H$ the pull-back of the hyperplane class and by $E_i$ the exceptional divisor over $q_i$. Then
$$
\Pic(X^r_k)=\ZZ H\oplus \ZZ E_1\oplus\cdots\oplus \ZZ E_k.
$$
The canonical class is $K_{X^r_k}=-(r+1)H+(r-1)\sum_{i=1}^kE_i$.

Recall that a normal projective variety $X$ is log Fano if there exists an effective $\QQ$-divisor $\Delta$ such that the pair $(X,\Delta)$ is klt and $-(K_X+\Delta)$ is ample. By \cite[Corollary 1.3.2]{BCHM2010}, log Fano varieties are Mori dream spaces.

Let us recall that the smooth Fano condition has a standard symplectic counterpart. A compact symplectic manifold $(M,\omega)$ is called monotone if $[\omega]=\lambda c_1(TM)$ for some $\lambda>0$. Thus, if $X$ is a smooth Fano variety and $\omega$ is a Kähler form whose class is proportional to $c_1(X)$, then $(X,\omega)$ is monotone. In this sense, monotonicity is the symplectic shadow of the positivity of the anticanonical class.

For a log pair $(X,\Delta)$ with $\Delta=\sum_i(1-\beta_i)D_i$, the corresponding symplectic condition would be the relative, or log, monotonicity relation
$$
[\omega]=\lambda\left(c_1(TM)-\sum_i(1-\beta_i)[D_i]\right),
\qquad \lambda>0,
$$
where the $D_i$ are symplectic divisors and $[D_i]$ denotes their cohomology classes. This is the natural analogue of the positivity of $-(K_X+\Delta)$, but we use it only as an analogy, not as a standard replacement for the algebraic notion of a log Fano pair.

The klt condition is even less intrinsic from the symplectic point of view. It is a birational condition, since it is defined by discrepancies on all resolutions. In the smooth normal crossing case, where $X$ is smooth and $\Delta=\sum_i(1-\beta_i)D_i$ has simple normal crossing support, it reduces to the coefficient condition $\beta_i>0$ for all $i$.

\begin{thm}\label{thm:AM-logfano} 
Let $X^r_k$ be the blow-up of $\PP^r$ at $k$ points in general position. Then $X^r_k$ is log Fano precisely in the following cases: 
\begin{enumerate} \item $r=2$ and $k\leq 8$; \item $r=3$ and $k\leq 7$; \item $r=4$ and $k\leq 8$; \item $r>4$ and $k\leq r+3$. 
\end{enumerate} In particular, in all dimensions $r\geq 2$, the blow-ups $X^r_{r+2}$ and $X^r_{r+3}$ are Mori dream spaces. 
\end{thm} 
\begin{proof} The log Fano classification is \cite[Theorem 1.3]{AraujoMassarenti2016}. The Mori dreamness statement follows from the finite generation results of \cite{CastravetTevelev2006}; it also follows from the log Fano statement and \cite[Corollary 1.3.2]{BCHM2010}. 
\end{proof}

\subsection*{The case $k=r+2$}

Let
$$
Y_r:=X^r_{r+2}=\Bl_{q_1,\ldots,q_{r+2}}\PP^r.
$$
We write a real divisor class on $Y_r$ as
$$
D=yH+\sum_{i=1}^{r+2}x_iE_i.
$$
The following theorem gives the effective cone, the movable cone, and the Mori chamber decomposition of $Y_r$.

\begin{thm}\label{thm:MCD-rplus2}
The variety $Y_r$ is a Mori dream space. Its effective cone is defined by the inequalities
$$
y+x_i\geq 0,\qquad i=1,\ldots,r+2,
$$
and
$$
ry+\sum_{i=1}^{r+2}x_i\geq 0.
$$
The Mori chamber decomposition of $\Eff(Y_r)$ is induced by the following hyperplane arrangement:
$$
(2-s)y-\sum_{i\in I}x_i=0,
\qquad
I\subset{1,\ldots,r+2},\quad \sharp I=s-1,
$$
and
$$
(r-s+1)y-\sum_{i\in I}x_i+\sum_{i=1}^{r+2}x_i=0,
\qquad
I\subset{1,\ldots,r+2},\quad \sharp I=s,
$$
where $2\leq s\leq (r+3)/2$.

Moreover, the movable cone is defined by the inequalities above, together with
$$
x_i\leq 0,\qquad i=1,\ldots,r+2,
$$
and
$$
(r-1)y-\sum_{i\in I}x_i+\sum_{i=1}^{r+2}x_i\geq 0,
\qquad
I\subset{1,\ldots,r+2},\quad \sharp I=2.
$$
All small $\QQ$-factorial modifications of $Y_r$ are smooth. Adjacent chambers in $\Mov(Y_r)$ are separated by one of the hyperplanes above and the corresponding birational map is a flip replacing a projective space $\PP^{s-2}$ by a projective space $\PP^{r+1-s}$.
\end{thm}

\begin{proof}
This is the description recalled in \cite[Definition 2.31 and Theorem 2.32]{BolognesiMassarenti2021}, based on the finite generation results of \cite{CastravetTevelev2006,AraujoMassarenti2016}.
\end{proof}

\subsection*{The case $k=r+3$}

Let 
$$Z_r:=X^r_{r+3}=\Bl_{q_1,\ldots,q_{r+3}}\PP^r.$$ 

We write a real divisor class on $Z_r$ as $D=yH+\sum_{i=1}^{r+3}x_iE_i$. The effective cone and its Mori chamber decomposition are described by a radial projection to the demi-hypercube.

First we recall the extremal rays of $\Eff(Z_r)$. Let $I\subset \{1,\ldots,r+3\}$ and assume that $\sharp I^c=2a+1$ is odd. Set
$$
E_I:=aH-a\sum_{i\in I}E_i-(a-1)\sum_{i\in I^c}E_i.
$$
For $a=0$, this yields the exceptional divisor $E_i$, with $I^c=\{i\}$. For $a\geq 1$, the class $E_I$ is the strict transform of the cone with vertex $\left\langle(q_i\mid i\in I)\right\rangle$ over the $(a-1)$-secant variety of the rational normal curve obtained by projecting the unique rational normal curve through $q_1,\ldots,q_{r+3}$.

\begin{thm}\label{thm:Eff-rplus3}
The variety $Z_r$ is a Mori dream space. Moreover, $\Eff(Z_r)$ is generated by the classes $E_I$, where $I\subset \{1,\ldots,r+3\}$ and $\sharp I^c$ is odd.
\end{thm}

\begin{proof}
This is the description recalled in \cite[Paragraph 3.3]{AraujoMassarenti2016}, using \cite{CastravetTevelev2006,Mukai2005}.
\end{proof}

Set
$$
\delta(D)=(r+1)y+\sum_{i=1}^{r+3}x_i.
$$
On the projectivized effective cone, define
$$
\varphi_i(D)=\frac{y+x_i}{\delta(D)},\qquad i=1,\ldots,r+3,
$$
and write $\varphi=(\varphi_1,\ldots,\varphi_{r+3})$. Thus $\varphi$ is constant on rays of $\Eff(Z_r)$. If $\xi_J\in\{0,1\}^{r+3}$ is the vertex whose $i$-th coordinate is $1$ exactly for $i\in J$, then
$$
\varphi(E_I)=\xi_{I^c}.
$$
Hence $\varphi$ identifies the projectivized effective cone with the demi-hypercube
$$
\Delta=\operatorname{Conv}\left(\xi_J\mid \sharp J \text{ is odd}\right)\subset [0,1]^{r+3}.
$$
For $I\subset \{1,\ldots,r+3\}$ and $\alpha=(\alpha_1,\ldots,\alpha_{r+3})$, set
$$
H_I(\alpha)=\sum_{j\notin I}\alpha_j+\sum_{i\in I}(1-\alpha_i).
$$
Then $\Delta$ is cut out by the inequalities $0\leq \alpha_i\leq 1$ and $H_I(\alpha)\geq 1$ for all even subsets $I$.

We now describe the Mori chamber decomposition explicitly. Let $\mathscr A$ be the hyperplane arrangement in $\RR^{r+3}$ given by
$$
H_I(\alpha)=s,\qquad
2\leq s\leq \frac{r+3}{2},\qquad
\sharp I\not\equiv s \pmod 2.
$$
The arrangement $\mathscr A$, together with the boundary faces of $\Delta$, induces a polyhedral subdivision of $\Delta$. Since $\varphi$ is radial, this subdivision determines a fan on $\Eff(Z_r)$: for every cell $C$ of the subdivision of $\Delta$, the corresponding cone is
$$
\sigma_C=\overline{\left\{D\in \Eff(Z_r)\mid \delta(D)>0,\ \varphi(D)\in C\right\}}.
$$
The cones $\sigma_C$ are precisely the cones of the Mori chamber decomposition of $\Eff(Z_r)$.

\begin{thm}\label{thm:MCD-rplus3}
The Mori chamber decomposition of $\Eff(Z_r)$ is the fan obtained as the cone over the subdivision of $\Delta$ induced by $\mathscr A$. In particular, the maximal Mori chambers are the cones $\sigma_C$, where $C$ runs through the maximal cells of
$$
\Delta\setminus \bigcup_{I,s}\{H_I(\alpha)=s\}.
$$
The movable cone is the cone over the polytope
$$
\Pi=\varphi(\Mov(Z_r))=
\left\{\alpha\in [0,1]^{r+3}\mid H_I(\alpha)\geq 2
\text{ for all odd } I\right\}.
$$
Thus the movable chambers are exactly the cones over the maximal cells of the induced subdivision of $\Pi$.

All small $\QQ$-factorial modifications of $Z_r$ are smooth. If two adjacent chambers inside $\Mov(Z_r)$ are separated by a wall $H_I(\alpha)=s$, with $3\leq s\leq (r+3)/2$ and $\sharp I\not\equiv s\pmod 2$, then the corresponding birational map flips a $\PP^{s-2}$ into a $\PP^{r+1-s}$.

The boundary walls of $\Mov(Z_r)$ are of two types. The first type is $\alpha_i=0$ or $\alpha_i=1$; the corresponding contraction is a $\PP^1$-bundle. The second type is $H_I(\alpha)=2$, with $\sharp I$ odd; the corresponding contraction is the blow-up of a smooth point, and the exceptional divisor is the image of $E_{I^c}$.
\end{thm}

\begin{proof}
The chamber decomposition and the wall-crossing description are those of \cite[Propositions 2 and 3]{Mukai2005}; and also in \cite[Theorem 3.4]{AraujoMassarenti2016}.
\end{proof}

\section{Configurations of points, trivial parabolic bundles, quivers and additive polygons}\label{Sec2}

The aim of this section is to identify several standard incarnations of the same moduli problem. Starting from weighted configurations of points on $\PP^1$, we pass through Hassett--GIT models, fixed-trivial-bundle parabolic structures, quiver moduli, and additive polygon spaces. We also keep track of the corresponding birational model of the blow-up $Y_{n-3}=\Bl_{n-1}\PP^{n-3}$, determined by its Mori chamber decomposition.

\subsection*{Configurations of points and the fixed trivial bundle}

Let $\bb=(b_1,\ldots,b_n)\in \QQ_{>0}^n$. After multiplying all $b_i$ by the same positive integer, we may assume that $b_i\in \ZZ_{>0}$ for all $i$. We set $|\bb|=b_1+\cdots+b_n$ and consider the linearized line bundle
$$
L_{\bb}:=\OO_{\PP^1}(b_1)\boxtimes\cdots\boxtimes \OO_{\PP^1}(b_n)
$$
on $(\PP^1)^n$, with respect to the diagonal action of $\PGL_2$. We will denote the corresponding GIT quotient by
$$
(\PP^1)^n_{\bb}:=(\PP^1)^n\quot_{\bb}\PGL_2.
$$
Thus $(\PP^1)^n_{\bb}$ is the good quotient of the semistable locus $((\PP^1)^n)^{ss}(\bb)$.

The Hilbert--Mumford criterion says that a configuration $x=(x_1,\ldots,x_n)$ is $\bb$-semistable if and only if, for every point $q\in\PP^1$, one has
$$
\sum_{x_i=q}b_i\leq \frac{|\bb|}{2}.
$$
It is $\bb$-stable if and only if all these inequalities are strict. This is the standard stability criterion for weighted configurations of points on $\PP^1$. We now normalize the weights by setting
$$
a_i=\frac{2b_i}{|\bb|}.
$$
Then $0<a_i$ and $\sum_i a_i=2$. We write $\avec=(a_1,\ldots,a_n)$. The vector $\avec$ lies on the boundary of the usual Hassett domain. Following Hassett's limiting construction and the standard GIT interpretation, we will denote by $\cM_{0,\avec}^{\mathrm{GIT}}$ the corresponding Hassett--GIT model. Concretely, $\cM_{0,\avec}^{\mathrm{GIT}}$ is the compactification of $M_{0,n}$ in which markings indexed by a subset $I\subset [n]$ are allowed to collide precisely when
$$
a_I:=\sum_{i\in I}a_i\leq 1.
$$
The equality $a_i=2b_i/|\bb|$ shows that this condition is the same as $b_I:=\sum_{i\in I}b_i\leq \frac{|\bb|}{2}$. We refer to \cite{Hassett2003} for details on moduli of curves with weighted marked points.  

\begin{Proposition}\label{prop:git-hassett-fixed}
There is a natural identification $(\PP^1)^n_{\bb}\cong \cM_{0,\avec}^{\mathrm{GIT}}$.
\end{Proposition}

\begin{proof}
A smooth point of $M_{0,n}$ is an ordered configuration of $n$ distinct points on an abstract rational curve, modulo automorphisms of the curve. After choosing an isomorphism of the curve with $\PP^1$, this is the same as an ordered configuration in $(\PP^1)^n$, modulo the diagonal action of $\PGL_2$.

The compactifications are controlled by the same collision rule. Indeed, by the Hilbert--Mumford criterion, a subset $I$ of markings is allowed to collide in the GIT quotient precisely when $b_I\leq |\bb|/2$. After normalization this is exactly the Hassett--GIT condition $a_I\leq 1$. Hence the semistable replacement rule on the GIT side and the collision rule on the Hassett--GIT side agree.
\end{proof}

We now define the fixed-trivial-bundle parabolic moduli problem. Fix pairwise distinct points $p_1,\ldots,p_n\in\PP^1$, and let $E_0=\OO_{\PP^1}^{\oplus 2}$. A quasi-parabolic structure on $E_0$ consists of a choice of lines
$$
V_i\subset (E_0)_{p_i},\qquad i=1,\ldots,n,
$$
where $(E_0)_{p_i}$ denotes the fiber of $E_0$ over $p_i$. After fixing a trivialization of $E_0$, each projective fiber $\PP((E_0)_{p_i})$ is identified with $\PP^1$. Therefore the space of quasi-parabolic structures on $E_0$ is $(\PP^1)^n$.

The automorphism group of $E_0$ is $\GL_2$. Scalar automorphisms act trivially on the projectivized fibers, hence the effective automorphism group is $\PGL_2$.

We now define the fixed-trivial-bundle moduli problem. Let $E_0=\OO_{\PP^1}^{\oplus 2}$. A parabolic structure on $E_0$ at the points $p_1,\ldots,p_n$ is a choice of one-dimensional subspaces
$$
V_i\subset (E_0)_{p_i},\qquad i=1,\ldots,n.
$$
Thus the parameter space of parabolic structures on $E_0$ is
$$
\Par(E_0):=\prod_{i=1}^n \PP((E_0)_{p_i}).
$$
After choosing a trivialization $E_0\cong \OO_{\PP^1}\otimes \CC^2$, we identify each projective fiber $\PP((E_0)_{p_i})$ with $\PP^1$. Hence
$$
\Par(E_0)\cong (\PP^1)^n.
$$
The automorphism group of $E_0$ is $\GL_2$. Scalar automorphisms act trivially on the projective fibers, so the effective group acting on $\Par(E_0)$ is $\GL_2/\CC^*=\PGL_2$. Under the identification $\Par(E_0)\cong(\PP^1)^n$, this is the diagonal action of $\PGL_2$ on $(\PP^1)^n$. We will use the weight vector $\bb=(b_1,\ldots,b_n)$ to linearize this action by
$$
L_{\bb}:=\OO_{\PP^1}(b_1)\boxtimes\cdots\boxtimes\OO_{\PP^1}(b_n).
$$
A parabolic structure $(V_1,\ldots,V_n)$ on $E_0$ is called $\bb$-semistable if the corresponding point of $(\PP^1)^n$ is semistable for this linearization.

More concretely, a degree zero line subbundle of $E_0$ is the same as a constant line $\ell\subset \CC^2$, that is
$$
L=\OO_{\PP^1}\otimes \ell\subset \OO_{\PP^1}\otimes \CC^2.
$$
For such a line subbundle, the condition $V_i=L_{p_i}$ means that the $i$-th parabolic direction is equal to the point $[\ell]\in\PP^1$. Therefore the Hilbert--Mumford criterion says that $(V_1,\ldots,V_n)$ is $\bb$-semistable if and only if, for every degree zero line subbundle $L\subset E_0$, one has
$$
\sum_{V_i=L_{p_i}} b_i\leq \frac{|\bb|}{2}.
$$
It is $\bb$-stable if all these inequalities are strict. We define
$$
\Mtriv_{\bb}:=\Par(E_0)^{ss}(\bb)\quot \PGL_2.
$$
Using the trivialization of $E_0$, this becomes
$$
\Mtriv_{\bb}\cong ((\PP^1)^n)^{ss}(\bb)\quot \PGL_2=(\PP^1)^n_{\bb}.
$$
This identification is independent of the chosen trivialization, since changing the trivialization only composes the identification $\Par(E_0)\cong(\PP^1)^n$ with the diagonal action of an element of $\PGL_2$.

\begin{Proposition}\label{prop:fixed-trivial-git}
There is a natural identification $\Mtriv_{\bb}\cong (\PP^1)^n_{\bb}$.
\end{Proposition}

\begin{proof}
By definition, $\Mtriv_{\bb}$ is the GIT quotient of the semistable locus in the parameter space of parabolic structures on $E_0$. The choice of a trivialization $E_0\cong \OO_{\PP^1}\otimes\CC^2$ identifies this parameter space with $(\PP^1)^n$. Under this identification, the action of $\Aut(E_0)/\CC^*$ becomes the diagonal action of $\PGL_2$, and the linearization is exactly $L_{\bb}$. Hence the quotient is $(\PP^1)^n_{\bb}$. Changing the trivialization changes the identification with $(\PP^1)^n$ by an element of $\PGL_2$. Therefore the resulting quotient is canonically independent of this choice.
\end{proof}

\subsection*{Additive polygons}

Let $S^2_{b_i}\subset \RR^3$ be the sphere of radius $b_i$. We define the additive polygon space by
$$
\Padd_{\bb}:=
\left\{(u_1,\ldots,u_n)\in S^2_{b_1}\times\cdots\times S^2_{b_n}
\mid u_1+\cdots+u_n=0\right\}/SO(3).
$$
Thus a point of $\Padd_{\bb}$ is a closed oriented polygonal chain in $\RR^3$ with side lengths $b_1,\ldots,b_n$, modulo rotations. The vectors $u_i$ are the side vectors of the polygon, and the equation $u_1+\cdots+u_n=0$ is the closing condition.

\begin{Definition}\label{def:generic-length}
A length vector $\bb=(b_1,\ldots,b_n)\in \RR_{>0}^n$ is called generic if $b_I\neq \frac{|\bb|}{2}$ for every non-empty proper subset $I\subset [n]$, where $b_I=\sum_{i\in I}b_i$. Equivalently, no subset of sides has total length equal to the total length of its complement.
\end{Definition}

If $\bb$ is generic in the sense of Definition \ref{def:generic-length}, then $0$ is a regular value of the moment map
$$
S^2_{b_1}\times\cdots\times S^2_{b_n}\longrightarrow \RR^3,
\qquad
(u_1,\ldots,u_n)\longmapsto u_1+\cdots+u_n.
$$
Therefore the closing condition $u_1+\cdots+u_n=0$ imposes three independent real equations, and since $SO(3)$ is $3$-dimensional we have that $\Padd_{\bb}$ has real dimension $2n-3-3=2(n-3)$.

\begin{Example}[A unit additive pentagon]\label{ex:unit-additive-pentagon}
Consider the length vector $\bb=(1,1,1,1,1)$. The following picture represents five unit vectors $u_1,\ldots,u_5\in S^2\subset\RR^3$ satisfying $u_1+\cdots+u_5=0$. Thus they define a point of the additive polygon space $\Padd_{\bb}$. The gray arrows are the vectors $u_i$ from the origin to the unit sphere, while the black segments join their endpoints and give a geometric visualization of the corresponding pentagonal configuration on $S^2$:

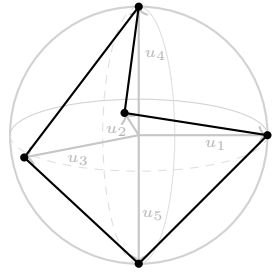
\begin{center}
\begin{tikzpicture}[scale=1.7, line cap=round, line join=round]



\coordinate (O)  at (0,0);
\coordinate (U1) at (1,0);
\coordinate (U2) at (-0.1103,0.1732);
\coordinate (U3) at (-0.8897,-0.1732);
\coordinate (U4) at (0,1);
\coordinate (U5) at (0,-1);

\draw[gray!35, thick] (O) circle (1);
\draw[gray!25, dashed] (-1,0) arc (180:360:1 and 0.28);
\draw[gray!35] (1,0) arc (0:180:1 and 0.28);
\draw[gray!25] (0,-1) arc (-90:90:0.28 and 1);
\draw[gray!25, dashed] (0,1) arc (90:270:0.28 and 1);

\draw[->, gray!45, thick] (O) -- (U1);
\draw[->, gray!45, thick] (O) -- (U2);
\draw[->, gray!45, thick] (O) -- (U3);
\draw[->, gray!45, thick] (O) -- (U4);
\draw[->, gray!45, thick] (O) -- (U5);

\node[gray!60] at (0.60,-0.08) {\tiny{$u_1$}};
\node[gray!60] at (-0.17,0.03) {\tiny{$u_2$}};
\node[gray!60] at (-0.47,-0.19) {\tiny{$u_3$}};
\node[gray!60] at (0.13,0.62) {\tiny{$u_4$}};
\node[gray!60] at (0.11,-0.62) {\tiny{$u_5$}};

\draw[black, thick] (U1) -- (U2) -- (U4) -- (U3) -- (U5) -- (U1);

\fill[black] (U1) circle (0.9pt);
\fill[black] (U2) circle (0.9pt);
\fill[black] (U3) circle (0.9pt);
\fill[black] (U4) circle (0.9pt);
\fill[black] (U5) circle (0.9pt);

\end{tikzpicture}
\end{center}
Under the identification $S^2\cong \PP^1$, each vector $u_i$ determines a point of $\PP^1$. Thus the polygon determines a configuration $([u_1],\ldots,[u_5])\in(\PP^1)^5$. Here $[u_i]$ is the point of $\PP^1$ corresponding to the endpoint of the vector $u_i$ on the unit sphere. Note that these are not the vertices of the polygonal chain obtained by placing the side vectors consecutively; they are the endpoints of the side vectors viewed as points of the coadjoint orbit $S^2$.
\end{Example}

\begin{Proposition}\label{prop:additive-git}
There is a natural homeomorphism $\Padd_{\bb}\cong (\PP^1)^n_{\bb}$. On the stable locus this homeomorphism is an isomorphism of real analytic orbifolds.
\end{Proposition}
\begin{proof}
The Lie algebra $\mathfrak{su}(2)$ consists of the traceless skew-Hermitian $2\times 2$ complex matrices: 
$$
\mathfrak{su}(2)
=
\left\{
A\in M_2(\CC)
\mid
A^*=-A,\ \operatorname{tr}(A)=0
\right\},
$$
where $A^*=\overline{A}^{\,t}$ denotes the conjugate transpose. Every element of $\mathfrak{su}(2)$ has the form
$$
A=
\begin{pmatrix}
i a & z \\
-\overline{z} & -i a
\end{pmatrix},
\qquad
a\in\RR,\ z\in\CC.
$$

Its dual $\mathfrak{su}(2)^*$ is the real vector space of real linear functionals on $\mathfrak{su}(2)$. We identify $\mathfrak{su}(2)^*$ with $\mathfrak{su}(2)$, and hence with $\RR^3$, by means of the $\operatorname{Ad}$-invariant inner product
$$
\langle A,B\rangle=-\frac{1}{2}\operatorname{tr}(AB).
$$
Here $\operatorname{Ad}_g(A)=gAg^{-1}$, and the invariance means that $\langle \operatorname{Ad}_g(A),\operatorname{Ad}_g(B)\rangle
=
\langle A,B\rangle$ for every $g\in SU(2)$ and every $A,B\in\mathfrak{su}(2)$.

The coadjoint action is the dual action on $\mathfrak{su}(2)^*$. Thus, if $\xi\in\mathfrak{su}(2)^*$, then
$$
(g\cdot \xi)(A)=\xi(\operatorname{Ad}_{g^{-1}}(A)),
\qquad
A\in\mathfrak{su}(2).
$$
More explicitly, an element $X\in\mathfrak{su}(2)$ defines a linear functional $\xi_X\in\mathfrak{su}(2)^*$ by $
\xi_X(A)=\langle X,A\rangle$. The $\operatorname{Ad}$-invariance of the inner product means that $
\langle \operatorname{Ad}_g(X),\operatorname{Ad}_g(A)\rangle
=
\langle X,A\rangle$ for every $g\in SU(2)$ and every $X,A\in\mathfrak{su}(2)$. Therefore
$$
(g\cdot \xi_X)(A)
=
\xi_X(\operatorname{Ad}_{g^{-1}}(A))
=
\langle X,\operatorname{Ad}_{g^{-1}}(A)\rangle
=
\langle \operatorname{Ad}_g(X),A\rangle
=
\xi_{\operatorname{Ad}_g(X)}(A).
$$
Hence, under the identification $\mathfrak{su}(2)^*\cong\mathfrak{su}(2)$, the coadjoint action coincides with the adjoint action. Now choose the basis
\begin{equation}\label{basis}
E_1=
\begin{pmatrix}
0 & 1\\
-1 & 0
\end{pmatrix},
\qquad
E_2=
\begin{pmatrix}
0 & i\\
i & 0
\end{pmatrix},
\qquad
E_3=
\begin{pmatrix}
i & 0\\
0 & -i
\end{pmatrix}.
\end{equation}
It identifies $\mathfrak{su}(2)$ with $\RR^3$ by $
x_1E_1+x_2E_2+x_3E_3
\longleftrightarrow
(x_1,x_2,x_3)$. The adjoint action preserves the inner product, so it gives a homomorphism
$$
SU(2)\longrightarrow SO(3),
\qquad
g\longmapsto \operatorname{Ad}_g.
$$
The kernel is the center ${\pm I}$, and the image is all of $SO(3)$. Thus $
SU(2)/{\pm I}\cong SO(3)$. Consequently, after the identifications $\mathfrak{su}(2)^*\cong\mathfrak{su}(2)\cong\RR^3$, the coadjoint action of $SU(2)$ becomes the usual rotation action of $SO(3)$ on $\RR^3$. In particular, if $\xi\in\mathfrak{su}(2)^*$ has norm $b$, then its coadjoint orbit is the sphere $SU(2)\cdot \xi=S^2_b\subset\RR^3$. This is the standard identification between the coadjoint representation of $SU(2)$ and the rotation representation of $SO(3)$ \cite[Chapter 1]{Kirillov2004} and \cite[Chapter 2]{GuilleminSternberg1990}.

Fix an element $\xi_i\in\mathfrak{su}(2)^*$ of norm $b_i$. Its coadjoint orbit is $O_{b_i}:=SU(2)\cdot \xi_i\subset \mathfrak{su}(2)^*$. Since the coadjoint action is the rotation action under the above identification, this orbit is the sphere $O_{b_i}=S^2_{b_i}\subset \RR^3$ of radius $b_i$. Thus the spheres appearing in the definition of $\Padd_{\bb}$ are precisely coadjoint orbits of $SU(2)$.

Every coadjoint orbit carries a canonical symplectic form, the Kostant--Kirillov--Souriau form. For $A\in\mathfrak{su}(2)$, the infinitesimal adjoint action is
$$
\operatorname{ad}_A(B)=[A,B],
\qquad
B\in\mathfrak{su}(2).
$$
The infinitesimal coadjoint action is the dual map
$$
\operatorname{ad}^*_A:\mathfrak{su}(2)^*\longrightarrow \mathfrak{su}(2)^*
$$
defined by
$$
(\operatorname{ad}^*_A\xi)(B)
=
-\xi([A,B]),
\qquad
\xi\in\mathfrak{su}(2)^*,\ B\in\mathfrak{su}(2).
$$
Thus the tangent space to the coadjoint orbit $\mathcal O_{b_i}=SU(2)\cdot \xi$ at $\xi$ is
$$
T_\xi\mathcal O_{b_i}
=
\left\{\operatorname{ad}^*_A\xi\mid A\in\mathfrak{su}(2)\right\}.
$$
At a point $\xi\in O_{b_i}$, tangent vectors to the orbit are of the form
$\operatorname{ad}^*_A(\xi)$, with $A\in\mathfrak{su}(2)$. With our convention for the coadjoint action, the Kostant--Kirillov--Souriau form is
$$
\omega_\xi(\operatorname{ad}^*_A\xi,\operatorname{ad}^*_B\xi)
=
\xi([A,B]).
$$
This is the standard symplectic form on coadjoint orbits; see \cite[Chapter 1]{Kirillov2004} and \cite[Chapter 3]{GuilleminSternberg1990}.

We now compare the Kostant--Kirillov--Souriau form with the Fubini--Study form. We write the coadjoint orbit as $
O_{b_i}:=SU(2)\cdot \xi_i$. Under the identification $\mathfrak{su}(2)^*\cong \RR^3$, this is the sphere $S^2_{b_i}$ of radius $b_i$.

The orbit $O_{b_i}$ is also naturally identified with $\PP^1$. Indeed, the stabilizer in $SU(2)$ of a non-zero element $\xi_i\in\mathfrak{su}(2)^*$ is a maximal torus $U(1)\subset SU(2)$, and hence $O_{b_i}\cong SU(2)/U(1)\cong \PP^1$. The second identification is the usual one between the homogeneous space $SU(2)/U(1)$ and the complex projective line.

Let $\omega_{\mathrm{KKS},i}$ be the Kostant--Kirillov--Souriau form on $O_{b_i}$. We normalize the Fubini--Study form $\omega_{\mathrm{FS}}$ on $\PP^1$ by requiring
$$
[\omega_{\mathrm{FS}}]=c_1(\OO_{\PP^1}(1)).
$$
With this normalization, the KKS form on the orbit through an element of norm $b_i$ satisfies $
[\omega_{\mathrm{KKS},i}]
=
b_i[\omega_{\mathrm{FS}}]
=
c_1(\OO_{\PP^1}(b_i))$. We now compute the cohomology class of the Kostant--Kirillov--Souriau form on $O_{b_i}$ explicitly. With respect to the inner product $\langle A,B\rangle$ the basis \eqref{basis} is orthonormal. Moreover,
$$
[E_1,E_2]=2E_3,\qquad
[E_2,E_3]=2E_1,\qquad
[E_3,E_1]=2E_2.
$$
Let $X_i=b_iE_3$. Under the identification $\mathfrak{su}(2)^*\cong\mathfrak{su}(2)\cong\RR^3$, the coadjoint orbit through $X_i$ is $
O_{b_i}=SU(2)\cdot X_i=S^2_{b_i}$. At the point $X_i=b_iE_3$, the tangent space to the orbit is generated by the vectors
$$
v_1=\operatorname{ad}^*_{E_1}(X_i)=[E_1,X_i]=-2b_iE_2 \text{ and } v_2=\operatorname{ad}^*_{E_2}(X_i)=[E_2,X_i]=2b_iE_1.
$$
With our convention, the Kostant--Kirillov--Souriau form is $
\omega_{\mathrm{KKS},X_i}
\left(
\operatorname{ad}^*_A(X_i),
\operatorname{ad}^*_B(X_i)
\right)
=
\langle X_i,[A,B]\rangle$. Therefore
$$
\omega_{\mathrm{KKS},X_i}(v_1,v_2) =
\langle b_iE_3,[E_1,E_2]\rangle =
\langle b_iE_3,2E_3\rangle = 2b_i.
$$
Let $\sigma_{b_i}$ be the Euclidean area form on the sphere $S^2_{b_i}$, with the orientation induced by the outward normal. Since $v_1=-2b_iE_2$ and $v_2=2b_iE_1$, we have $
\sigma_{b_i,X_i}(v_1,v_2)=4b_i^2$. Hence, at $X_i$, we have $\omega_{\mathrm{KKS},i}
=\frac{1}{2b_i}\sigma_{b_i}$. Both forms are $SU(2)$-invariant, so the equality holds on the whole orbit $O_{b_i}=S^2_{b_i}$. It follows that
$$
\int_{O_{b_i}}\omega_{\mathrm{KKS},i} =
\frac{1}{2b_i}\operatorname{Area}(S^2_{b_i}) =
\frac{1}{2b_i}\cdot 4\pi b_i^2 =
2\pi b_i.
$$
Now identify $O_{b_i}$ with $\PP^1$. We normalize the Fubini--Study form $\omega_{\mathrm{FS}}$ by $
\int_{\PP^1}\omega_{\mathrm{FS}}=2\pi$. Thus $\frac{[\omega_{\mathrm{FS}}]}{2\pi}=c_1(\OO_{\PP^1}(1))$. Since $H^2(\PP^1,\RR)$ is one-dimensional, the integral computation above gives $
[\omega_{\mathrm{KKS},i}]
=
b_i[\omega_{\mathrm{FS}}]$. Hence,
$$
\frac{[\omega_{\mathrm{KKS},i}]}{2\pi}
=
b_i c_1(\OO_{\PP^1}(1))
=
c_1(\OO_{\PP^1}(b_i)).
$$
This is the precise meaning of saying that $O_{b_i}\cong\PP^1$ is the Hamiltonian $SU(2)$-manifold associated with the linearized line bundle $\OO_{\PP^1}(b_i)$. The group $SU(2)$ acts on $\PP^1$ in the standard way, the line bundle $\OO_{\PP^1}(b_i)$ carries the induced $SU(2)$-linearization, and an invariant Hermitian metric on $\OO_{\PP^1}(b_i)$ has curvature form representing $
c_1(\OO_{\PP^1}(b_i))$. By the computation above, this is the same cohomology class as the KKS form on $O_{b_i}$. Moreover, the corresponding moment map is, up to the fixed normalization, the inclusion $
O_{b_i}\hookrightarrow \mathfrak{su}(2)^*$. Thus the Hamiltonian data coming from the coadjoint orbit $O_{b_i}$ agree with the Hamiltonian data coming from the polarized projective variety $(\PP^1,\OO_{\PP^1}(b_i))$.

Now consider the product
$$
M:=S^2_{b_1}\times\cdots\times S^2_{b_n}
=
O_{b_1}\times\cdots\times O_{b_n}.
$$
We equip $M$ with the product symplectic form $
\omega:=\omega_1+\cdots+\omega_n$, where $\omega_i$ is the Kostant--Kirillov--Souriau form on the $i$-th coadjoint orbit. The diagonal action of $SU(2)$ on $M$ is Hamiltonian. Its moment map is the sum of the orbit inclusions:
$$
\mu:M\longrightarrow \mathfrak{su}(2)^*\cong\RR^3,\qquad
\mu(u_1,\ldots,u_n)=u_1+\cdots+u_n.
$$
This is the standard moment map for a product of coadjoint orbits \cite[Section 2]{Kirwan1984}. The symplectic quotient of $M$ at zero is
$$
\mu^{-1}(0)/SU(2)
=
\left\{(u_1,\ldots,u_n)\in S^2_{b_1}\times\cdots\times S^2_{b_n}
\mid u_1+\cdots+u_n=0\right\}/SU(2).
$$
The center ${\pm I}\subset SU(2)$ acts trivially on every coadjoint orbit, hence the quotient by $SU(2)$ is the same as the quotient by $SO(3)$. Therefore
$$
\mu^{-1}(0)/SU(2)=\Padd_{\bb}.
$$
This is the usual realization of Euclidean polygon spaces as symplectic reductions of products of spheres \cite{KapovichMillson1996,HausmannKnutson1998}.

On the other hand, via the identifications $S^2_{b_i}\cong\PP^1$, the Hamiltonian $SU(2)$-manifold $M$ is the compact Kähler manifold $(\PP^1)^n$ equipped with the product Kähler form whose cohomology class is the first Chern class of
$$
L_{\bb}:=\OO_{\PP^1}(b_1)\boxtimes\cdots\boxtimes\OO_{\PP^1}(b_n).
$$
The complexification of $SU(2)$ is $SL_2(\CC)$. We now apply the Kempf--Ness theorem \cite{KempfNess1979}, \cite[Chapter 8]{MumfordFogartyKirwan1994}, \cite[Section 2]{Kirwan1984}. In the present setting, this theorem says that if a smooth projective variety is acted on by a compact group $K$, with complexification $G=K_{\CC}$, and the Kähler form comes from a compatible ample linearization, then the symplectic quotient $\mu^{-1}(0)/K$ is naturally homeomorphic to the GIT quotient by $G$. Applied to $K=SU(2)$, $G=SL_2(\CC)$, and the linearized variety $((\PP^1)^n,L_{\bb})$, it gives a natural homeomorphism
$$
\mu^{-1}(0)/SU(2)
\cong
(\PP^1)^n\quot_{\bb}SL_2(\CC).
$$
Finally, the center ${\pm I}\subset SL_2(\CC)$ acts trivially on $\PP^1$, and hence trivially on $(\PP^1)^n$. Therefore the invariant sections and the semistable loci for the $SL_2(\CC)$-linearization and for the induced $PGL_2$-linearization agree. Hence the GIT quotient by $SL_2(\CC)$ is the same as the GIT quotient by $PGL_2=SL_2(\CC)/{\pm I}$ that is
$$
(\PP^1)^n\quot_{\bb}SL_2(\CC)
=
(\PP^1)^n\quot_{\bb}\PGL_2
=
(\PP^1)^n_{\bb}.
$$
Therefore $\Padd_{\bb}$ and $(\PP^1)^n_{\bb}$ are homeomorphic. On the stable locus, the Kempf--Ness homeomorphism identifies the symplectic quotient with the stable GIT quotient in the real analytic category. Since stabilizers are finite on the stable locus, the identification is an isomorphism of real analytic orbifolds \cite[Section 2]{Kirwan1984}, \cite[Chapter 8]{MumfordFogartyKirwan1994}.
\end{proof}

\begin{Remark}
Assume that $\bb$ is generic and that $\Padd_{\bb}$ is non-empty. Then $0$ is a regular value of the moment map and by Marsden--Weinstein reduction \cite{MarsdenWeinstein1974}, \cite[Chapter 5]{McDuffSalamon2017}, $\Padd_{\bb}$ is a smooth compact symplectic manifold. Moreover, each sphere $S^2_{b_i}$ is identified with $\PP^1$ endowed with a multiple of the Fubini--Study form, and the diagonal $SU(2)$-action is Hamiltonian and holomorphic. Thus Kähler reduction endows $\Padd_{\bb}$ with a natural Kähler structure \cite{GuilleminSternberg1982}, \cite[Section 2]{Kirwan1984}. Hence $\Padd_{\bb}$ has a structure of Kähler manifold of dimension $n-3$.  The algebraic structure is obtained by comparing this Kähler quotient with the GIT quotient. By Kempf--Ness, $
\Padd_{\bb}\cong (\PP^1)^n\quot_{\bb}\PGL_2$. Thus, through this identification, $\Padd_{\bb}$ becomes a complex projective variety. In the generic case this projective variety is smooth; in the non-generic case one obtains instead a singular projective GIT quotient, corresponding on the symplectic side to a singular stratified symplectic quotient.
\end{Remark}

\subsection*{The quiver viewpoint}

We now describe the quiver interpretation of the fixed-trivial-bundle side. Let $Q$ be the star-shaped quiver with one central vertex $0$, outer vertices $1,\ldots,n$, and one arrow $i\longrightarrow 0$ for each $i=1,\ldots,n$. We take the dimension vector
$$
\mathbf v=(2,1,\ldots,1),
$$
that is, the vector space at the central vertex is $\CC^2$, and the vector space at each outer vertex is $\CC$.

Thus a representation of $Q$ of dimension vector $\mathbf v$ is an $n$-tuple of linear maps
$$
q_i:\CC\longrightarrow \CC^2,\qquad i=1,\ldots,n.
$$
After choosing the standard basis of $\CC$, each $q_i$ is the same as a vector in $\CC^2$. Therefore
$$
\operatorname{Rep}(Q,\mathbf v)
=
\bigoplus_{i=1}^n\operatorname{Hom}(\CC,\CC^2)
\cong
(\CC^2)^n.
$$
The base-change group is $
G_{\mathbf v}:=
\left(GL_2(\CC)\times(\CC^*)^n\right)/\CC^*$, where the diagonal copy of $\CC^*$ acts trivially. Explicitly, an element represented by $(g,t_1,\ldots,t_n)$ acts by $q_i\longmapsto gq_it_i^{-1}$. We will use the character determined by the weight vector $\bb$ that is
$$
\theta_{\bb}=(-|\bb|,2b_1,\ldots,2b_n),
\qquad
|\bb|=b_1+\cdots+b_n.
$$
This satisfies $2(-|\bb|)+\sum_{i=1}^n2b_i=0$. So it is a legitimate stability parameter for the dimension vector $\mathbf v$. The corresponding quiver moduli space is the GIT quotient
$$
\mathcal M_Q(\bb):=
\operatorname{Rep}(Q,\mathbf v)^{ss}(\theta_{\bb})\quot G_{\mathbf v}.
$$
This is the moduli space of $\theta_{\bb}$-semistable representations of the quiver $Q$ with dimension vector $\mathbf v$, in the sense of the standard GIT construction of quiver moduli \cite{King1994,Nakajima1994}.

The relation with configurations of points on $\PP^1$ is explicit. The factors $(\CC^*)^n$ rescale the non-zero vectors $q_i$, so quotienting by them projectivizes each vector:
$$
(q_1,\ldots,q_n)
\longmapsto
([q_1],\ldots,[q_n])\in(\PP^1)^n.
$$
After this quotient, the remaining group is $PGL_2$, acting diagonally on $(\PP^1)^n$. The character $\theta_{\bb}$ induces the product linearization $\OO_{\PP^1}(2b_1)\boxtimes\cdots\boxtimes\OO_{\PP^1}(2b_n)$. Multiplying all weights by $2$ does not change the GIT quotient. Hence $
\mathcal M_Q(\bb)
\cong
(\PP^1)^n\quot_{\bb}PGL_2
=
(\PP^1)^n_{\bb}$. By the additive dictionary, this is also the additive polygon space: $
\mathcal M_Q(\bb)
\cong
(\PP^1)^n_{\bb}
\cong
\Padd_{\bb}$.

The same identification can be seen symplectically. Let
$$
K_{\mathbf v}:=
\left(U(2)\times U(1)^n\right)/U(1)
$$
be the maximal compact subgroup of $G_{\mathbf v}$. The vector space $(\CC^2)^n$ has its standard Kähler form, and the action of $K_{\mathbf v}$ is Hamiltonian. The Kähler quotient of $(\CC^2)^n$ by $K_{\mathbf v}$ with parameter determined by $\bb$ is homeomorphic to the above GIT quotient by the Kempf--Ness theorem \cite{KempfNess1979,Kirwan1984,MumfordFogartyKirwan1994}.

More concretely, the $U(1)^n$-moment-map equations fix the norms of the vectors $q_i$. After quotienting by $U(1)^n$, each non-zero vector $q_i\in\CC^2$ gives a point of $\PP^1$, or, under the Hopf map, a point of the sphere $S^2_{b_i}$. The remaining $SU(2)$-moment map is precisely $
(u_1,\ldots,u_n)\mapsto u_1+\cdots+u_n$.
Thus the final Kähler quotient is
$$
\left\{(u_1,\ldots,u_n)\in S^2_{b_1}\times\cdots\times S^2_{b_n}
\mid u_1+\cdots+u_n=0\right\}/SO(3)
=
\Padd_{\bb}.
$$
Therefore the quiver moduli space $\mathcal M_Q(\bb)$ is a projective GIT quotient, hence a complex projective variety. When the stability parameter is generic, it is smooth and its Kähler quotient description gives its natural Kähler form.

\begin{Example}\label{Add_Q_5}
For $n=5$, the quiver is the following star-shaped quiver:
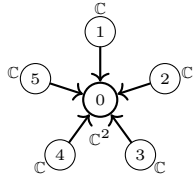
\begin{center}
\begin{tikzpicture}[scale=0.45, line cap=round, line join=round]
\node[circle,draw,thick,minimum size=4.5mm,inner sep=0pt] (O) at (0,0) {\tiny{$0$}};
\node at (0,-1.1) {\tiny{$\CC^2$}};

\node[circle,draw,minimum size=4mm,inner sep=0pt] (A1) at (0,2) {\tiny{$1$}};
\node[circle,draw,minimum size=4mm,inner sep=0pt] (A2) at (1.9,0.62) {\tiny{$2$}};
\node[circle,draw,minimum size=4mm,inner sep=0pt] (A3) at (1.18,-1.62) {\tiny{$3$}};
\node[circle,draw,minimum size=4mm,inner sep=0pt] (A4) at (-1.18,-1.62) {\tiny{$4$}};
\node[circle,draw,minimum size=4mm,inner sep=0pt] (A5) at (-1.9,0.62) {\tiny{$5$}};

\node at (0,2.67) {\tiny{$\CC$}};
\node at (2.55,0.80) {\tiny{$\CC$}};
\node at (1.75,-1.98) {\tiny{$\CC$}};
\node at (-1.77,-1.98) {\tiny{$\CC$}};
\node at (-2.55,0.80) {\tiny{$\CC$}};

\draw[->,thick] (A1) -- (O);
\draw[->,thick] (A2) -- (O);
\draw[->,thick] (A3) -- (O);
\draw[->,thick] (A4) -- (O);
\draw[->,thick] (A5) -- (O);
\end{tikzpicture}
\end{center}
A representation consists of five vectors $q_1,\ldots,q_5\in\CC^2$. After quotienting by the five scalar actions, these give five points $[q_1],\ldots,[q_5]\in\PP^1$; after quotienting by $PGL_2$, one obtains the same GIT quotient as before.
\end{Example}

We now pass to the doubled quiver. Let $\overline Q$ be obtained from $Q$ by adding the opposite arrow $0\longrightarrow i$ for each $i=1,\ldots,n$. Thus a representation of $\overline Q$ of dimension vector $\mathbf v=(2,1,\ldots,1)$ consists of pairs $q_i:\CC\longrightarrow \CC^2$ and $p_i:\CC^2\longrightarrow \CC$, $i=1,\ldots,n$. Here the maps $q_i$ correspond to the original arrows $i\to 0$, while the maps $p_i$ correspond to the opposite arrows $0\to i$.

Since $\operatorname{Rep}(Q,\mathbf v)$ is a vector space, its cotangent bundle is
$$
T^*\operatorname{Rep}(Q,\mathbf v)=\operatorname{Rep}(Q,\mathbf v)\oplus \operatorname{Rep}(Q,\mathbf v)^*.
$$
In our case $\operatorname{Rep}(Q,\mathbf v)=\bigoplus_{i=1}^n\operatorname{Hom}(\CC,\CC^2)$. Using the trace pairing, $\operatorname{Rep}(Q,\mathbf v)^*=\bigoplus_{i=1}^n\operatorname{Hom}(\CC^2,\CC)$. Therefore
$$
T^*\operatorname{Rep}(Q,\mathbf v)=
\bigoplus_{i=1}^n
\left(
\operatorname{Hom}(\CC,\CC^2)
\oplus
\operatorname{Hom}(\CC^2,\CC)
\right).
$$
This is exactly $\operatorname{Rep}(\overline Q,\mathbf v)=T^*\operatorname{Rep}(Q,\mathbf v)$.

The vector space $T^*\operatorname{Rep}(Q,\mathbf v)$ carries its standard flat hyperkähler structure. The compact group $K_{\mathbf v}:=(U(2)\times U(1)^n)/U(1)$ acts by unitary changes of bases at the vertices. Its Lie algebra is $\mathfrak{k}_{\mathbf v}:=(\mathfrak{u}(2)\oplus \mathfrak{u}(1)^n)/\mathfrak{u}(1)$. The complexification of $K_{\mathbf v}$ is $G_{\mathbf v}:=(GL_2(\CC)\times(\CC^*)^n)/\CC^*$, with Lie algebra $\mathfrak{g}_{\mathbf v}:=(\mathfrak{gl}_2(\CC)\oplus \CC^n)/\CC$.

The action of $K_{\mathbf v}$ on $T^*\operatorname{Rep}(Q,\mathbf v)$ is Hamiltonian for the hyperkähler structure. Hence there are moment maps
$$
\mu_{\RR}:T^*\operatorname{Rep}(Q,\mathbf v)\longrightarrow \mathfrak{k}_{\mathbf v}^*,
\qquad
\mu_{\CC}:T^*\operatorname{Rep}(Q,\mathbf v)\longrightarrow \mathfrak{g}_{\mathbf v}^*.
$$
Here $\mathfrak{k}_{\mathbf v}^*$ is the real dual of the Lie algebra of $K_{\mathbf v}$, while $\mathfrak{g}_{\mathbf v}^*$ is the complex dual of the Lie algebra of $G_{\mathbf v}$.

Let $\lambda_{\bb}\in\mathfrak{k}_{\mathbf v}^*$ be the real moment-map parameter determined by $\bb$. The notation $T^*\operatorname{Rep}(Q,\mathbf v)\hkquot_{\bb}K_{\mathbf v}$ means the hyperkähler quotient at real level $\lambda_{\bb}$ and complex level $0$, that is
$$
T^*\operatorname{Rep}(Q,\mathbf v)\hkquot_{\bb}K_{\mathbf v}
:=
\left(
\mu_{\RR}^{-1}(\lambda_{\bb})
\cap
\mu_{\CC}^{-1}(0)
\right)/K_{\mathbf v}.
$$
We will denote this hyperkähler quotient by $\mathfrak X_{\bb}:=T^*\operatorname{Rep}(Q,\mathbf v)\hkquot_{\bb}K_{\mathbf v}$. This is the hyperpolygon space associated with $\bb$ \cite{Konno2002,GodinhoMandini2013}.

In one of its complex structures, the same hyperkähler quotient is the Nakajima quiver variety
$$
\mathfrak X_{\bb}
=
\mu_{\CC}^{-1}(0)^{ss}(\theta_{\bb})\quot G_{\mathbf v}.
$$
Here $\theta_{\bb}$ is the stability parameter defined above, and the quotient is the GIT quotient by $G_{\mathbf v}$ \cite{Nakajima1994,Konno2002,GodinhoMandini2013}. With respect to this complex structure, $\mathfrak X_{\bb}$ carries the holomorphic symplectic form induced by the canonical holomorphic symplectic form on $T^*\operatorname{Rep}(Q,\mathbf v)$.

The same complex symplectic manifold has a parabolic-Higgs interpretation. More precisely, $\mathfrak X_{\bb}$ is identified with the moduli space of certain rank two parabolic Higgs bundles on $\PP^1$ with fixed marked points and weights determined by $\bb$. Under this identification the variables $p_i$ encode the Higgs-field data. Moreover, the holomorphic symplectic form on the hyperpolygon space agrees with the natural Higgs symplectic form on the parabolic-Higgs moduli space \cite{GodinhoMandini2013,BiswasFlorentinoGodinhoMandini2015}. Thus the identification with parabolic Higgs bundles is a holomorphic symplectomorphism.

Let us write the complex moment-map equations explicitly. A point of $T^*\operatorname{Rep}(Q,\mathbf v)$ is a collection $(q_i,p_i)_{i=1}^n$, with $q_i:\CC\longrightarrow \CC^2$ and $p_i:\CC^2\longrightarrow \CC$. The composition $q_ip_i:\CC^2\longrightarrow \CC^2$ is an endomorphism of $\CC^2$, while $p_iq_i:\CC\longrightarrow \CC$ is a scalar endomorphism. The complex moment-map equation at the central vertex is $\sum_{i=1}^n(q_ip_i)_0=0$, where
$$
(q_ip_i)_0
=
q_ip_i-\frac{1}{2}\operatorname{tr}(q_ip_i)\operatorname{Id}_{\CC^2}
$$
is the traceless part of $q_ip_i$. The equations at the outer vertices are $p_iq_i=0$, $i=1,\ldots,n$. Thus a point of $\mathfrak X_{\bb}$ is represented by pairs $(q_i,p_i)$ satisfying these equations, modulo the action of $G_{\mathbf v}$ and the stability condition $\theta_{\bb}$.

The ordinary polygon space sits inside the hyperpolygon space as the zero-section locus. Indeed, if $p_1=\cdots=p_n=0$, then the complex moment-map equations are automatically satisfied, and we recover the original quiver representation data $(q_1,\ldots,q_n)$. Hence there is a natural inclusion
$$
\Padd_{\bb}
\cong
\mathcal M_Q(\bb)
\subset
\mathfrak X_{\bb}.
$$
Here $\mathcal M_Q(\bb)$ is the quiver moduli space of the original star-shaped quiver. In this sense, the hyperpolygon space is the cotangent, or hyperkähler, enhancement of the additive polygon space.

For generic $\bb$, the hyperpolygon space $\mathfrak X_{\bb}$ is a smooth hyperkähler manifold. In particular, it is Kähler in each complex structure. In the complex-algebraic structure above, it is a smooth quasi-projective variety. It is generally non-compact, and therefore it is not projective in general. The compact projective subvariety obtained by setting all $p_i=0$ is the ordinary additive polygon space.

\begin{Example}
For $n=5$, the doubled quiver is:
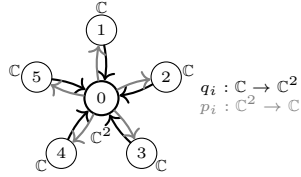
\begin{center}
\begin{tikzpicture}[scale=0.45, line cap=round, line join=round]
\node[circle,draw,thick,minimum size=4.5mm,inner sep=0pt] (O) at (0,0) {\tiny{$0$}};
\node at (0,-1.1) {\tiny{$\CC^2$}};

\node[circle,draw,minimum size=4mm,inner sep=0pt] (A1) at (0,2) {\tiny{$1$}};
\node[circle,draw,minimum size=4mm,inner sep=0pt] (A2) at (1.9,0.62) {\tiny{$2$}};
\node[circle,draw,minimum size=4mm,inner sep=0pt] (A3) at (1.18,-1.62) {\tiny{$3$}};
\node[circle,draw,minimum size=4mm,inner sep=0pt] (A4) at (-1.18,-1.62) {\tiny{$4$}};
\node[circle,draw,minimum size=4mm,inner sep=0pt] (A5) at (-1.9,0.62) {\tiny{$5$}};

\node at (0,2.67) {\tiny{$\CC$}};
\node at (2.55,0.80) {\tiny{$\CC$}};
\node at (1.75,-1.98) {\tiny{$\CC$}};
\node at (-1.77,-1.98) {\tiny{$\CC$}};
\node at (-2.55,0.80) {\tiny{$\CC$}};

\draw[->,thick] (A1) to[bend left=10] (O);
\draw[->,thick,gray] (O) to[bend left=10] (A1);

\draw[->,thick] (A2) to[bend left=10] (O);
\draw[->,thick,gray] (O) to[bend left=10] (A2);

\draw[->,thick] (A3) to[bend left=10] (O);
\draw[->,thick,gray] (O) to[bend left=10] (A3);

\draw[->,thick] (A4) to[bend left=10] (O);
\draw[->,thick,gray] (O) to[bend left=10] (A4);

\draw[->,thick] (A5) to[bend left=10] (O);
\draw[->,thick,gray] (O) to[bend left=10] (A5);

\node at (4.35,0.35) {\tiny{$q_i:\CC\to\CC^2$}};
\node[gray] at (4.35,-0.20) {\tiny{$p_i:\CC^2\to\CC$}};
\end{tikzpicture}
\end{center}
\end{Example}

The relation between polygons and hyperpolygons is direct. The ordinary quiver moduli space $\mathcal M_Q(\bb)$ embeds into the hyperpolygon space by setting all cotangent variables equal to zero:
$$
(q_1,\ldots,q_n)\longmapsto (q_1,\ldots,q_n,p_1=0,\ldots,p_n=0).
$$
The complex moment-map equations are then automatically satisfied, and the stability condition reduces to the original one. Under the identifications above, this yields the inclusion
$$
\Padd_{\bb}
\cong
\mathcal M_Q(\bb)
\subset
\mathfrak X_{\bb}.
$$
Thus the hyperpolygon space is a cotangent, or hyperkähler, enhancement of the additive polygon space. In the parabolic-Higgs interpretation, the variables $p_i$ encode the Higgs-field data, and the ordinary polygon space is the zero-Higgs-field locus. With respect to the holomorphic symplectic structure described above, the identification with the corresponding parabolic-Higgs moduli space is a symplectomorphism \cite{GodinhoMandini2013,BiswasFlorentinoGodinhoMandini2015}.

The structural difference is the following. The ordinary quiver moduli space $\mathcal M_Q(\bb)$ is compact and projective; it is the same projective variety as $(\PP^1)^n_{\bb}$. The hyperpolygon space $\mathfrak X_{\bb}$ is hyperkähler, hence Kähler in each of its complex structures, and it is a smooth quasi-projective variety for generic $\bb$. It is generally non-compact and therefore not projective. Its compact subvariety obtained by setting $p_i=0$ is precisely the additive polygon space.

\subsection*{Normalization of additive weights.}
We recall how the weights are normalized in the additive dictionary. Let $\bb=(b_1,\ldots,b_n)\in\QQ_{>0}^n$, set $|\bb|=\sum_i b_i$ and $\bar b=|\bb|/2$. The GIT quotient $(\PP^1)^n_{\bb}=((\PP^1)^n)^{ss}(\bb)\quot\PGL_2$ depends only on the ray $\QQ_{>0}\bb$: by the Hilbert--Mumford criterion, a configuration is semistable precisely when every subset $I$ of points colliding at the same point of $\PP^1$ satisfies $\sum_{i\in I}b_i\leq\bar b$. Thus replacing $\bb$ by $\lambda\bb$, with $\lambda>0$, does not change the semistable locus, but only rescales the polarization.

To compare with Hassett weights, we choose the representative of this ray having total weight $2$, namely $\avec=(2b_1/|\bb|,\ldots,2b_n/|\bb|)$. Then $\sum_i a_i=2$, and the GIT condition becomes $\sum_{i\in I}a_i\leq 1$, which is the collision rule for the Hassett--GIT model $\cM_{0,\avec}^{\mathrm{GIT}}$. Thus $\avec$ is not a new stability datum: it is the normalized representative of the GIT ray of $\bb$.

The fixed-trivial-bundle parabolic quotient, the star-shaped quiver quotient and the additive polygon space are naturally written using any representative of the same ray. More generally, if $A=(\alpha_1,\ldots,\alpha_n)$ has total weight $S=\sum_i\alpha_i$, then taking $b_i=\alpha_i$ gives the stability condition $\sum_{i\in I}\alpha_i\leq S/2$. For parabolic structures on the fixed bundle $\OO_{\PP^1}^{\oplus 2}$, this is the usual slope condition, since the parabolic slope of the whole bundle is $S/2$. For additive polygons, scaling all lengths by the same positive constant does not change the associated GIT quotient; it only rescales the symplectic form under the Kempf--Ness identification.

For the star-shaped quiver with dimension vector $(2,1,\ldots,1)$, the corresponding quiver stability parameter is $\theta_{\bb}=(-|\bb|,2b_1,\ldots,2b_n)$, where the first entry corresponds to the central vertex. It satisfies $\theta_{\bb}\cdot(2,1,\ldots,1)=0$. With the convention that $\theta$-semistability means $\theta(W)\leq 0$ for every subrepresentation $W$, the subrepresentation determined by a line in $\CC^2$ containing the vectors indexed by $I$ gives the condition $-|\bb|+2\sum_{i\in I}b_i\leq 0$, that is $\sum_{i\in I}b_i\leq|\bb|/2$. Hence the quiver stability is exactly the GIT stability above. Moreover, $\theta_{\lambda\bb}=\lambda\theta_{\bb}$, so the quiver moduli space depends only on the same ray. In normalized Hassett coordinates one may use $\theta_{\avec}=(-2,2a_1,\ldots,2a_n)$, or equivalently $\vartheta_{\avec}=(-1,a_1,\ldots,a_n)$. For a non-normalized vector $A=(\alpha_i)$ with total weight $S$, the corresponding additive quiver parameter is $\theta_A=(-S,2\alpha_1,\ldots,2\alpha_n)$, or equivalently $\vartheta_A=(-S/2,\alpha_1,\ldots,\alpha_n)$.

In particular, if one starts from weights $A=(\alpha_i)$ with $S=\sum_i\alpha_i>2$, then the additive GIT quotient, the fixed-trivial-bundle parabolic quotient, the quiver moduli space and the additive polygon space are obtained by taking $\bb=A$. The associated Hassett--GIT weights are instead the radial normalization $\widehat A=(2\alpha_1/S,\ldots,2\alpha_n/S)$. Thus $A$ itself should not be confused with the Hassett weights of the same additive model: Hassett stability for $A$ would use the collision rule $\sum_{i\in I}\alpha_i\leq 1$, whereas the additive GIT, fixed-trivial-bundle, quiver and polygon stability condition is $\sum_{i\in I}\alpha_i\leq S/2$. Conversely, starting from $\bb$, every vector $A=\lambda\bb$ gives the same additive model; if one wants $\sum_i\alpha_i>2$, then $\lambda>2/|\bb|$, and if one also wants $0<\alpha_i<1$ for all $i$, one must impose $\lambda b_i<1$ for all $i$. Finally, the divisor class on the blow-up scales as $D_{\lambda\bb,n}=\lambda D_{\bb,n}$, hence the Mori chamber and the associated Mori model depend only on the ray of $\bb$.

We summarize the additive dictionary in the following statement, making explicit both the moduli-theoretic identifications and the Mori chamber of $Y_{n-3}$ which produces the corresponding birational model.

\begin{thm}\label{thm:additive-summary}
Let $\bb=(b_1,\ldots,b_n)\in \QQ_{>0}^n$, and assume that the semistable locus is non-empty. Set $|\bb|=\sum_i b_i$, $\bar b=|\bb|/2$, and
$$
\avec=\left(\frac{2b_1}{|\bb|},\ldots,\frac{2b_n}{|\bb|}\right).
$$
Fix the distinguished label $n$, and set $r=n-3$. Let
$$
Y_{n-3}:=\Bl_{q_1,\ldots,q_{n-1}}\PP^{n-3}.
$$
On $Y_{n-3}$ consider the $\QQ$-divisor class
$$
D_{\bb,n}:=(\bar b-b_n)H-\sum_{i=1}^{n-1}(\bar b-b_n-b_i)E_i.
$$
Let $C_{\bb,n}$ be the Mori chamber of $\Eff(Y_{n-3})$ containing $D_{\bb,n}$ in its relative interior, or the corresponding face if $D_{\bb,n}$ lies on a wall. Let
$$
\phi_{\bb,n}:Y_{n-3}\dashrightarrow Y_{\bb,n}
$$
be the Mori model associated with $C_{\bb,n}$, that is $Y_{\bb,n}=\operatorname{Proj}R(Y_{n-3},D_{\bb,n})$. Then there are natural identifications
$$
Y_{\bb,n}
\cong
(\PP^1)^n_{\bb}
\cong
\cM_{0,\avec}^{\mathrm{GIT}}
\cong
\Mtriv_{\bb}
\cong
\mathcal M_Q(\bb)
\cong
\Padd_{\bb}.
$$
Here $(\PP^1)^n_{\bb}=((\PP^1)^n)^{ss}(\bb)\quot\PGL_2$ is taken with respect to $L_{\bb}=\OO_{\PP^1}(b_1)\boxtimes\cdots\boxtimes\OO_{\PP^1}(b_n)$, the Hassett--GIT weights are $\avec$, the fixed-trivial-bundle parabolic quotient uses the same linearization $L_{\bb}$, the quiver stability parameter is
$$
\theta_{\bb}=(-|\bb|,2b_1,\ldots,2b_n),
$$
and $\Padd_{\bb}$ is the additive polygon space with side lengths $b_1,\ldots,b_n$. More precisely, the first five spaces are naturally isomorphic as projective varieties. The last identification is the Kempf--Ness identification between the GIT quotient $(\PP^1)^n_{\bb}$ and the symplectic quotient defining $\Padd_{\bb}$. Thus $\Padd_{\bb}$ carries the natural projective structure associated with this GIT/Kempf--Ness realization.
\end{thm}

\begin{proof}
The model $Y_{\bb,n}$ is the GIT model associated with the divisor $D_{\bb,n}$ by the associated-points construction and the VGIT interpretation of the Mori chamber decomposition of $Y_{n-3}$; \cite{Kumar2003,BolognesiMassarenti2021}. In the notation of Theorem \ref{thm:MCD-rplus2}, the relevant chamber is exactly the chamber $C_{\bb,n}\subset\Eff(Y_{n-3})$ containing $D_{\bb,n}$.

The identification $(\PP^1)^n_{\bb}\cong \cM_{0,\avec}^{\mathrm{GIT}}$ is Proposition \ref{prop:git-hassett-fixed}. The identification $\Mtriv_{\bb}\cong(\PP^1)^n_{\bb}$ is Proposition \ref{prop:fixed-trivial-git}. The identification with $\mathcal M_Q(\bb)$ follows from the quiver construction above, since quotienting by the outer scalar factors projectivizes the vectors $q_i$, and the character $\theta_{\bb}$ induces the product linearization $L_{\bb}$ up to multiplication by a positive scalar. Finally, $\Padd_{\bb}\cong(\PP^1)^n_{\bb}$ is Proposition \ref{prop:additive-git}, via the symplectic reduction of the product of coadjoint orbits and the Kempf--Ness theorem.
\end{proof}

\subsection*{The blow-up chamber and wall-crossing to the Fano model}

We now single out the chamber which gives the original blow-up $Y_{n-3}=\Bl_{q_1,\ldots,q_{n-1}}\PP^{n-3}$. We use normalized weights $A=(a_1,\ldots,a_n)$ with $a_i>0$ and $\sum_i a_i=2$, and we fix the label $n$ distinguished. The chamber corresponding to the nef cone of $Y_{n-3}$ is
$$
\mathcal C_Y=
\left\{
A\mid a_i+a_n<1 \text{ for all } i=1,\ldots,n-1,\quad
a_i+a_j+a_n>1 \text{ for all }1\leq i<j\leq n-1
\right\}.
$$
Thus, if $\bb=(b_1,\ldots,b_n)$ is any positive vector and $a_i=2b_i/|\bb|$, then the additive model associated with $\bb$ is exactly the blow-up $Y_{n-3}$ if and only if
$$
b_i+b_n<\frac{|\bb|}{2}\text{ for all }i=1,\ldots,n-1,\qquad
b_i+b_j+b_n>\frac{|\bb|}{2}\text{ for all }1\leq i<j\leq n-1.
$$
In this chamber the GIT quotient $(\PP^1)^n_{\bb}$, the fixed-trivial-bundle parabolic quotient $\Mtriv_{\bb}$, the quiver moduli space with stability parameter $\theta_{\bb}=(-|\bb|,2b_1,\ldots,2b_n)$, and the additive polygon space $\Padd_{\bb}$ are all identified with $Y_{n-3}$. In modular terms, this chamber singles out the $n$-th point, or equivalently the $n$-th leg of the quiver, or the $n$-th side of the polygon.

The symmetric, or Fano, model is obtained from the symmetric weight
$A_F=(2/n,\ldots,2/n)$, equivalently from the equal length vector $\bb_F=(1,\ldots,1)$. We will denote it by $\Sigma_{n-3}:=(\PP^1)^n\quot \PGL_2$. If $n$ is odd, $A_F$ lies in a chamber; if $n$ is even, it lies on the symmetric wall. In either case, the corresponding GIT model is the Fano model in the additive Mori chamber decomposition.

To pass from $Y_{n-3}$ to $\Sigma_{n-3}$ one chooses a path of weights from $\mathcal C_Y$ to $A_F$. The walls crossed by this path are the hyperplanes $a_I=1$, where $a_I=\sum_{i\in I}a_i$ and $2\leq |I|\leq n-2$. Crossing such a wall is the VGIT wall-crossing for configurations of points on $\PP^1$, hence also the wall-crossing for the fixed-trivial-bundle parabolic quotient, for the quiver quotient, and for the additive polygon space. If $A^-$ and $A^+$ lie in the two adjacent chambers and $a_I<1$ on the $A^-$-side while $a_I>1$ on the $A^+$-side, then the corresponding birational map replaces a projective space $\PP^{|I|-2}$ with a projective space $\PP^{n-|I|-2}$. Reversing the crossing reverses the flip.

The weights themselves are modified simply by moving across the wall while keeping total weight $2$. If $A^0$ lies on the wall $a_I=1$, a convenient local normal direction is
$$
\eta_i=
\begin{cases}
1/|I| & \text{if } i\in I,\\
-1/|I^c| & \text{if } i\notin I.
\end{cases}
$$
Then $A^\pm=A^0\pm\epsilon\eta$, for $0<\epsilon\ll 1$, lie on the two sides of the wall and satisfy $\sum_i a_i^\pm=2$. On the quiver side the stability parameter changes from $\theta_{A^-}=(-2,2a_1^-,\ldots,2a_n^-)$ to $\theta_{A^+}=(-2,2a_1^+,\ldots,2a_n^+)$. On the parabolic side the parabolic weights are changed from $A^-$ to $A^+$. On the polygon side the side lengths are changed from $A^-$ to $A^+$, or from any positive scalar multiple of $A^-$ to the corresponding scalar multiple of $A^+$.

Equivalently, in non-normalized length coordinates, a wall is given by $b_I=|\bb|/2$. If the total length is kept fixed, one may cross the wall by increasing the lengths indexed by $I$ and decreasing the lengths indexed by $I^c$ proportionally, or vice versa. The chamber only depends on the ray of $\bb$, so this operation may be followed by any common positive rescaling of all lengths.

In polygon language, legths vectors lying on a wall are non-generic. In fact, if $\bb$ satisfies the wall equation  $b_I=|\bb|/2$ for some $I$, then there is a strictly semistable collinear polygon: the sides indexed by $I$ point in one direction, while the sides indexed by $I^c$ point in the opposite direction. Crossing the wall changes which of the two subsets is short. The wall-crossing is therefore a symplectic flip: on one side one resolves the collinear polygon by opening the $I$-block, and on the other side by opening the $I^c$-block, cf. \cite{Mandini2014}. The two exceptional loci are $\PP^{|I|-2}$ and $\PP^{n-|I|-2}$.

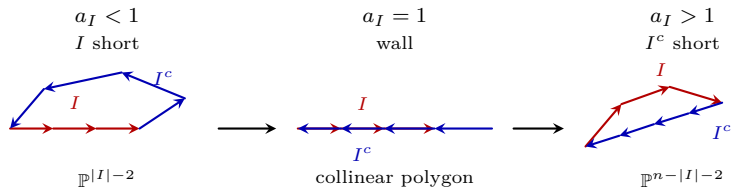
\begin{center}
\begin{tikzpicture}[scale=0.95, line cap=round, line join=round, >=stealth]

\node at (0,1.55) {\footnotesize{$a_I<1$}};
\node at (0,1.20) {\scriptsize{$I$ short}};
\coordinate (A0) at (-1.35,0);
\coordinate (A1) at (-0.75,0);
\coordinate (A2) at (-0.15,0);
\coordinate (A3) at (0.45,0);
\coordinate (A4) at (1.08,0.42);
\coordinate (A5) at (0.20,0.78);
\coordinate (A6) at (-0.88,0.55);
\draw[thick,red!70!black,->] (A0)--(A1);
\draw[thick,red!70!black,->] (A1)--(A2);
\draw[thick,red!70!black,->] (A2)--(A3);
\draw[thick,blue!70!black,->] (A3)--(A4);
\draw[thick,blue!70!black,->] (A4)--(A5);
\draw[thick,blue!70!black,->] (A5)--(A6);
\draw[thick,blue!70!black,->] (A6)--(A0);
\node[red!70!black] at (-0.45,0.36) {\scriptsize{$I$}};
\node[blue!70!black] at (0.78,0.70) {\scriptsize{$I^c$}};
\node at (0,-0.69) {\scriptsize{$\PP^{|I|-2}$}};

\node at (4.0,1.55) {\footnotesize{$a_I=1$}};
\node at (4.0,1.20) {\scriptsize{wall}};
\coordinate (B0) at (2.65,0);
\coordinate (B1) at (3.25,0);
\coordinate (B2) at (3.85,0);
\coordinate (B3) at (4.55,0);
\coordinate (B4) at (5.35,0);
\draw[thick,red!70!black,->] (B0)--(B1);
\draw[thick,red!70!black,->] (B1)--(B2);
\draw[thick,red!70!black,->] (B2)--(B3);
\draw[thick,blue!70!black,->] (B4)--(B3);
\draw[thick,blue!70!black,->] (B3)--(B2);
\draw[thick,blue!70!black,->] (B2)--(B1);
\draw[thick,blue!70!black,->] (B1)--(B0);
\node[red!70!black] at (3.55,0.35) {\scriptsize{$I$}};
\node[blue!70!black] at (3.55,-0.35) {\scriptsize{$I^c$}};
\node at (4.0,-0.69) {\scriptsize{collinear polygon}};

\node at (8.0,1.55) {\footnotesize{$a_I>1$}};
\node at (8.0,1.20) {\scriptsize{$I^c$ short}};
\coordinate (C0) at (6.65,-0.25);
\coordinate (C1) at (7.15,0.34);
\coordinate (C2) at (7.82,0.58);
\coordinate (C3) at (8.55,0.36);
\coordinate (C4) at (8.08,0.21);
\coordinate (C5) at (7.60,0.06);
\coordinate (C6) at (7.12,-0.10);
\draw[thick,red!70!black,->] (C0)--(C1);
\draw[thick,red!70!black,->] (C1)--(C2);
\draw[thick,red!70!black,->] (C2)--(C3);
\draw[thick,blue!70!black,->] (C3)--(C4);
\draw[thick,blue!70!black,->] (C4)--(C5);
\draw[thick,blue!70!black,->] (C5)--(C6);
\draw[thick,blue!70!black,->] (C6)--(C0);
\node[red!70!black] at (7.70,0.82) {\scriptsize{$I$}};
\node[blue!70!black] at (8.55,-0.05) {\scriptsize{$I^c$}};
\node at (8.0,-0.69) {\scriptsize{$\PP^{n-|I|-2}$}};

\draw[->,thick] (1.55,0) -- (2.35,0);
\draw[->,thick] (5.65,0) -- (6.35,0);

\end{tikzpicture}
\end{center}

A useful path from the blow-up chamber to the Fano model is obtained by taking $a_1=\cdots=a_{n-1}=x$ and $a_n=2-(n-1)x$. The blow-up chamber is the interval $1/(n-2)<x<1/(n-3)$. Moving toward the symmetric value $x=2/n$ crosses the walls $a_n+\sum_{j\in J}a_j=1$, with $J\subset\{1,\ldots,n-1\}$. Since this symmetric path meets several walls at the same time, one should perturb it slightly in order to obtain a genuine sequence of elementary flips. Each elementary step is one of the flips described above.

In quiver terms the same wall-crossing has the following interpretation. A representation of the star-shaped quiver with dimension vector $(2,1,\ldots,1)$ is a collection of vectors $q_i\in V=\CC^2$, one for each outer vertex. Given a line $L\subset V$ and a subset $I\subset\{1,\ldots,n\}$ such that $q_i\in L$ for all $i\in I$, there is a subrepresentation $W_{L,I}$ with central space $L$ and outer vertices indexed by $I$. For the normalized stability parameter $\theta_A=(-2,2a_1,\ldots,2a_n)$ one has $\theta_A(W_{L,I})=-2+2a_I$. Hence the wall $a_I=1$ is exactly the wall where $W_{L,I}$ has slope zero. On the side $a_I<1$, this subrepresentation is allowed by stability; on the side $a_I>1$, it destabilizes. At the wall, the strictly semistable objects are represented by two blocks supported on the line $L$ and on the quotient $V/L$, with dimension vectors $(1,\chi_I)$ and $(1,\chi_{I^c})$. Crossing the wall exchanges the two possible extension directions between these two stable factors. The exceptional loci are the projectivized extension spaces, of dimensions $|I|-2$ and $n-|I|-2$, matching the polygon flip described above.

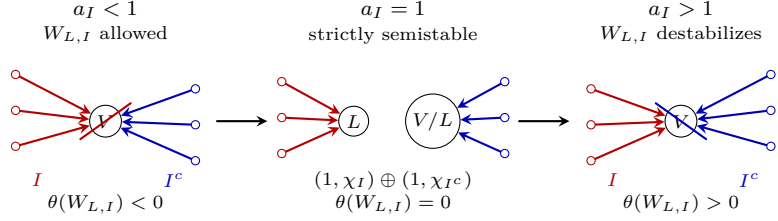
\begin{center}
\begin{tikzpicture}[scale=0.95, line cap=round, line join=round, >=stealth,
central/.style={circle,draw,inner sep=1.4pt},
outer/.style={circle,draw,inner sep=1.1pt},
redv/.style={red!70!black},
bluev/.style={blue!70!black}]

\node at (0,1.55) {\footnotesize{$a_I<1$}};
\node at (0,1.20) {\scriptsize{$W_{L,I}$ allowed}};
\node[central] (C1) at (0,0) {\scriptsize{$V$}};
\node[outer,redv] (R11) at (-1.25,0.65) {};
\node[outer,redv] (R12) at (-1.25,0.15) {};
\node[outer,redv] (R13) at (-1.25,-0.35) {};
\node[outer,bluev] (B11) at (1.25,0.45) {};
\node[outer,bluev] (B12) at (1.25,-0.05) {};
\node[outer,bluev] (B13) at (1.25,-0.55) {};
\draw[->,thick,redv] (R11)--(C1);
\draw[->,thick,redv] (R12)--(C1);
\draw[->,thick,redv] (R13)--(C1);
\draw[->,thick,bluev] (B11)--(C1);
\draw[->,thick,bluev] (B12)--(C1);
\draw[->,thick,bluev] (B13)--(C1);
\draw[redv,thick] (-0.35,-0.25)--(0.35,0.25);
\node[redv] at (-0.95,-0.82) {\scriptsize{$I$}};
\node[bluev] at (0.95,-0.82) {\scriptsize{$I^c$}};
\node at (0,-1.15) {\scriptsize{$\theta(W_{L,I})<0$}};

\node at (4.0,1.55) {\footnotesize{$a_I=1$}};
\node at (4.0,1.20) {\scriptsize{strictly semistable}};
\node[central] (L2) at (3.45,0) {\scriptsize{$L$}};
\node[central] (Q2) at (4.55,0) {\scriptsize{$V/L$}};
\node[outer,redv] (R21) at (2.45,0.55) {};
\node[outer,redv] (R22) at (2.45,0.05) {};
\node[outer,redv] (R23) at (2.45,-0.45) {};
\node[outer,bluev] (B21) at (5.55,0.55) {};
\node[outer,bluev] (B22) at (5.55,0.05) {};
\node[outer,bluev] (B23) at (5.55,-0.45) {};
\draw[->,thick,redv] (R21)--(L2);
\draw[->,thick,redv] (R22)--(L2);
\draw[->,thick,redv] (R23)--(L2);
\draw[->,thick,bluev] (B21)--(Q2);
\draw[->,thick,bluev] (B22)--(Q2);
\draw[->,thick,bluev] (B23)--(Q2);
\node at (4.0,-0.82) {\scriptsize{$(1,\chi_I)\oplus(1,\chi_{I^c})$}};
\node at (4.0,-1.15) {\scriptsize{$\theta(W_{L,I})=0$}};

\node at (8.0,1.55) {\footnotesize{$a_I>1$}};
\node at (8.0,1.20) {\scriptsize{$W_{L,I}$ destabilizes}};
\node[central] (C3) at (8,0) {\scriptsize{$V$}};
\node[outer,redv] (R31) at (6.75,0.45) {};
\node[outer,redv] (R32) at (6.75,-0.05) {};
\node[outer,redv] (R33) at (6.75,-0.55) {};
\node[outer,bluev] (B31) at (9.25,0.65) {};
\node[outer,bluev] (B32) at (9.25,0.15) {};
\node[outer,bluev] (B33) at (9.25,-0.35) {};
\draw[->,thick,redv] (R31)--(C3);
\draw[->,thick,redv] (R32)--(C3);
\draw[->,thick,redv] (R33)--(C3);
\draw[->,thick,bluev] (B31)--(C3);
\draw[->,thick,bluev] (B32)--(C3);
\draw[->,thick,bluev] (B33)--(C3);
\draw[bluev,thick] (7.65,0.25)--(8.35,-0.25);
\node[redv] at (7.05,-0.82) {\scriptsize{$I$}};
\node[bluev] at (8.95,-0.82) {\scriptsize{$I^c$}};
\node at (8,-1.15) {\scriptsize{$\theta(W_{L,I})>0$}};

\draw[->,thick] (1.55,0) -- (2.25,0);
\draw[->,thick] (5.75,0) -- (6.45,0);

\end{tikzpicture}
\end{center}

The colored line drawn inside the central space $V$ is not an additional datum of the quiver representation. It only indicates the line $L\subset V$ supporting the critical subrepresentation, and, after crossing the wall, the complementary quotient $V/L$ appearing in the associated graded object.

\begin{Example}
We describe explicitly the wall-crossing from the blow-up chamber to the Fano model in the case $n=7$. Let $A=(a_1,\ldots,a_7)$ with $a_i>0$ and $\sum_i a_i=2$, and keep the label $7$ distinguished. The chamber corresponding to the original blow-up $Y_4=\Bl_{q_1,\ldots,q_6}\PP^4$ is
$$
a_i+a_7<1\quad\text{for all }i=1,\ldots,6,\qquad
a_i+a_j+a_7>1\quad\text{for all }1\leq i<j\leq 6.
$$
The Fano model $\Sigma_4$ is the symmetric quotient $(\PP^1)^7\quot\PGL_2$, corresponding to the weight $A_F=(2/7,\ldots,2/7)$. Thus, moving from the blow-up chamber to the Fano chamber, one has to cross the fifteen walls
$$
H_{ij}: a_7+a_i+a_j=1,\quad 1\leq i<j\leq 6.
$$
The symmetric path $a_1=\cdots=a_6=x$, $a_7=2-6x$ crosses all these walls simultaneously at $x=1/4$; a generic small perturbation gives a sequence of fifteen elementary flips, one for each pair $\{i,j\}$.

Fix one such wall $H_{ij}$ and set $I=\{7,i,j\}$. Near the wall choose a point $A^0$ with $a^0_I=1$. The two adjacent chambers are obtained by
$$
a_h^\pm=
\begin{cases}
a_h^0\mp \epsilon/3 & \text{if }h\in I,\\
a_h^0\pm \epsilon/4 & \text{if }h\notin I,
\end{cases}
\qquad 0<\epsilon\ll 1.
$$
Then $\sum_h a_h^\pm=2$, $a_I^-=1+\epsilon$ and $a_I^+=1-\epsilon$. Thus the direction from the blow-up chamber to the Fano chamber is the direction in which the triple $\{7,i,j\}$ changes from long to short.

In terms of the star-shaped quiver, the stability parameter changes from $\theta^-=\theta_{A^-}=(-2,2a_1^-,\ldots,2a_7^-)$ to $\theta^+=\theta_{A^+}=(-2,2a_1^+,\ldots,2a_7^+)$. The critical subrepresentations are those supported by a line $L\subset\CC^2$ containing the three vectors $q_7,q_i,q_j$. For such a subrepresentation $W_{L,I}$ one has $\theta_A(W_{L,I})=-2+2a_I$. Hence $W_{L,I}$ is destabilizing on the blow-up side, has slope zero on the wall, and is allowed on the Fano side. At the wall the strictly semistable objects have associated graded object supported on the two blocks $I=\{7,i,j\}$ and $I^c$. The flip exchanges the two extension directions between these two stable factors. Since $|I|=3$ and $|I^c|=4$, the elementary wall-crossing exchanges the projective spaces $\PP^1$ and $\PP^2$; in the direction from $Y_4$ to $\Sigma_4$, the $\PP^2$-side is replaced by the $\PP^1$-side.

Thus a path from $\Bl_6\PP^4$ to the Fano model is obtained by choosing an ordering of the pairs $\{i,j\}\subset\{1,\ldots,6\}$ and crossing successively the walls
$
a_7+a_i+a_j=1.
$
At the wall corresponding to $\{i,j\}$, the weights of the three legs $7,i,j$ are decreased by $\epsilon/3$, while the weights of the four complementary legs are increased by $\epsilon/4$, keeping the total weight equal to $2$. On the quiver side this is exactly the change of stability parameter described above; on the polygon side it is the wall-crossing where the three sides $7,i,j$ pass from being a long block to being a short block.
\end{Example}

\section{Parabolic bundles, multiplicative polygons and multiplicative quivers}\label{Sec3}

We now pass to the blow-up of $\PP^{n-3}$ at $n$ general points and its relation with the moduli spaces of rank two parabolic bundles with trivial determinant, and the multiplicative polygon spaces.

\subsection*{The Bauer--Mukai chamber and the birational model}
Let $Z_n:=\Bl_{q_1,\ldots,q_n}\PP^{n-3}$, where $q_1,\ldots,q_n\in\PP^{n-3}$ are general points. This is the case $r=n-3$ and $k=r+3$ of Theorem \ref{thm:MCD-rplus3}. We write a divisor class on $Z_n$ as $D=yH+\sum_{i=1}^n x_iE_i$ and use the projection
$$
\varphi_i(D)=\frac{y+x_i}{(n-2)y+\sum_{j=1}^n x_j},
\qquad i=1,\ldots,n.
$$
The image of the projectivized effective cone is the demi-hypercube
$$
\Delta=\operatorname{Conv}\left(\xi_J\mid J\subset\{1,\ldots,n\},\ \sharp J \text{ is odd}\right)\subset[0,1]^n.
$$
For $I\subset\{1,\ldots,n\}$ and $\alpha=(\alpha_1,\ldots,\alpha_n)$, set
$$
H_I(\alpha)=\sum_{j\notin I}\alpha_j+\sum_{i\in I}(1-\alpha_i).
$$
The Mori chamber decomposition of $\Eff(Z_n)$ is the cone over the chamber decomposition of $\Delta$ induced by the hyperplanes
$$
H_I(\alpha)=s,\qquad
2\leq s\leq \frac{n}{2},\qquad
\sharp I\not\equiv s \pmod 2.
$$

Let $A=(a_1,\ldots,a_n)\in\Delta$ be a generic weight, that is, assume that $A$ does not lie on any wall of the above arrangement. Let $C_A$ be the chamber of $\Delta$ containing $A$, and let $\sigma_A\subset\Eff(Z_n)$ be the corresponding Mori chamber, that is
$$
\sigma_A=\overline{\left\{
D\in\Eff(Z_n)\mid
(n-2)y+\sum_{j=1}^n x_j>0,\ \varphi(D)\in C_A
\right\}}.
$$
Choose any $\QQ$-divisor $D_A$ in the relative interior of $\sigma_A$ and define the associated Mori model by
$$
Z_A:=\operatorname{Proj}\bigoplus_{m\geq 0}H^0(Z_n,\OO_{Z_n}(mD_A)).
$$
This model depends only on the chamber $C_A$, not on the particular choice of $D_A$.

\subsection*{Rank two parabolic bundles with trivial determinant}

Fix pairwise distinct points $p_1,\ldots,p_n\in\PP^1$ and a weight $A=(a_1,\ldots,a_n)\in\Delta$, with $0<a_i<1$. A rank two quasi-parabolic bundle with trivial determinant consists of a rank two vector bundle $E$ on $\PP^1$, an isomorphism $\det E\cong\OO_{\PP^1}$, and lines
$$
V_i\subset E_{p_i},\qquad i=1,\ldots,n.
$$
The bundle $E$ is not fixed. By B--G, it splits as $E\cong\OO_{\PP^1}(m)\oplus\OO_{\PP^1}(-m)$ for some $m\geq 0$. For a line subbundle $L\subset E$, define
$$
\epsilon_i(L)=
\begin{cases}
1 & \text{if } L_{p_i}=V_i,\\
0 & \text{if } L_{p_i}\neq V_i.
\end{cases}
$$
The parabolic slopes are
$$
\mu_A(E)=\frac{1}{2}\sum_{i=1}^n a_i,\qquad
\mu_A(L)=\deg L+\sum_{i=1}^n a_i\epsilon_i(L).
$$
The parabolic bundle $(E,V_1,\ldots,V_n)$ is $A$-semistable if $\mu_A(L)\leq\mu_A(E)$ for every line subbundle $L\subset E$, and it is $A$-stable if all these inequalities are strict.

Let $M_A$ be the coarse moduli space of $S$-equivalence classes of $A$-semistable rank two parabolic bundles with trivial determinant. If $A$ is generic, semistability equals stability and $M_A$ is a smooth projective variety of dimension $n-3$.

\begin{thm}\label{thm:BM-model}
For $A$ generic, the Mori model $Z_A$ is naturally isomorphic to the moduli space $M_A$ of $A$-stable rank two parabolic bundles on $\PP^1$ with trivial determinant.
\end{thm}

\begin{proof} 
In \cite[Section 1]{Mukai2005}, the moduli spaces $U(\alpha)$ of semistable parabolic rank two bundles on the $n$-pointed line are considered with fixed determinant. For a diagonal weight $\alpha=(a,\ldots,a)$, \cite[Proposition 1]{Mukai2005} shows that, if $$ \frac{1}{n-2}<a<\frac{1}{n-4}, $$ then $U(\alpha)$ is isomorphic to the blow-up $Z_n=\Bl_{q_1,\ldots,q_n}\PP^{n-3}$. Thus one chamber of the parabolic weight space gives the initial model $Z_n$. For general weights, \cite[Proposition 2]{Mukai2005}, based on \cite{Bauer1991}, gives the wall-crossing description. The walls are exactly $$ H_I(\alpha)=s,\qquad 2\leq s\leq \frac{n}{2},\qquad \sharp I\not\equiv s \pmod 2. $$ Inside a chamber, the isomorphism class of $U(\alpha)$ is constant. If two adjacent chambers are separated by a wall $H_I(\alpha)=s$, then the corresponding models are related by a blow-up of a smooth point for $s=2$, and by a flip replacing a $\PP^{s-2}$ with a $\PP^{n-s-2}$ for $3\leq s\leq n/2$. Moreover, \cite[Proposition 3]{Mukai2005} describes the boundary walls $\alpha_i=0$ and $\alpha_i=1$ as $\PP^1$-bundle contractions. 

On the other hand, Theorem \ref{thm:MCD-rplus3}, ultimately based on \cite[Propositions 2 and 3]{Mukai2005}, says that the Mori chamber decomposition of $\Eff(Z_n)$ is the cone over the same hyperplane arrangement, and that the movable cone is the cone over $\Pi$. In particular, the chambers in $\Pi$ correspond exactly to the small $\QQ$-factorial modifications of $Z_n$, and their wall-crossings are the same flips described above. It remains to identify the polarizations, not only the birational wall-crossings. By the GIT construction of parabolic moduli in \cite{MehtaSeshadri1980}, each $U(\alpha)$ carries a natural ample divisor, denoted $D_\alpha$ in \cite[Section 1]{Mukai2005}. 

Mukai proves that the class $D_\alpha$ varies linearly with $\alpha$ on each chamber and that, for a chamber $C$, the cone generated by the classes $D_\alpha$ with $\alpha\in C$ is the corresponding chamber in the movable cone of $Z_n$. Hence, if $A\in C_A$, the ample model of the semiample divisor $D_A$ on the corresponding small modification is precisely $U(A)$. Starting from the diagonal chamber, where $U(\alpha)\cong Z_n$, and moving to an adjacent chambers from that chamber to $C_A$, the parabolic wall-crossings and the Mori wall-crossings coincide at each step. Therefore the model obtained from $Z_n$ by the Mori chamber $\sigma_A$ is the same projective variety as the parabolic moduli space $U(A)$. With our notation $M_A=U(A)$, this yields $Z_A\cong M_A$. 
\end{proof}

\subsection*{Multiplicative polygons}

For $0<a_i<1$, let $\mathcal C_i\subset SU(2)$ be the conjugacy class of matrices with eigenvalues $\exp(\pi i a_i)$ and $\exp(-\pi i a_i)$. Define
$$
\Mult_A:=
\left\{
(g_1,\ldots,g_n)\in \mathcal C_1\times\cdots\times \mathcal C_n
\mid g_1\cdots g_n=1
\right\}/SU(2),
$$
where $SU(2)$ acts by simultaneous conjugation. 

This is the multiplicative analogue of the additive polygon space: the linear closing equation $u_1+\cdots+u_n=0$ is replaced by the group-valued closing equation $g_1\cdots g_n=1$. Equivalently, setting $h_0=1$ and $h_j=g_1\cdots g_j$, one gets a closed chain $h_0,h_1,\ldots,h_n=h_0$ in $SU(2)$ whose increments $h_{j-1}^{-1}h_j$ have prescribed conjugacy classes. Thus these are not Euclidean polygons, but closed polygons in the Lie group $SU(2)$ in the multiplicative sense.

\begin{Remark}
Assume that $0<a_i<1$ for all $i$. Then each conjugacy class $\mathcal C_i\subset SU(2)$ is the conjugacy class of a non-central element. Its stabilizer is a maximal torus $U(1)$, hence
$$
\mathcal C_i\cong SU(2)/U(1)\cong S^2
$$
and $\dim_{\RR}\mathcal C_i=2$. Therefore $\mathcal C_1\times\cdots\times\mathcal C_n$ has real dimension $2n$. The multiplicative closing condition $g_1\cdots g_n=1$ is one equation with values in $SU(2)$, so it cuts three real dimensions at regular points. The quotient by simultaneous conjugation by $SU(2)$ cuts another three real dimensions on the stable locus. Thus $\Mult_A$ has real dimension $2n-3-3=2(n-3)$.
\end{Remark}

\begin{Example}[A unit multiplicative pentagon]\label{ex:unit-multiplicative-pentagon}
Consider the central weight $A_F=(1/2,\ldots,1/2)$. Then each conjugacy class $\mathcal C_i\subset SU(2)$ consists of elements with eigenvalues $i$ and $-i$. Identifying $SU(2)$ with the unit quaternions, this is the conjugacy class of pure imaginary unit quaternions.

For instance, take
$$
g_1=\mathbf i,\qquad
g_2=\mathbf j,\qquad
g_3=\mathbf k,\qquad
g_4=\mathbf i,\qquad
g_5=\mathbf i.
$$
Then $g_1,\ldots,g_5\in\mathcal C_i$ and $g_1g_2g_3g_4g_5=\mathbf i\mathbf j\mathbf k\mathbf i\mathbf i=1$. Thus they define a point of $\Mult_{A_F}$. If $h_0=1$ and $h_j=g_1\cdots g_j$, then
$$
h_0=1,\quad h_1=\mathbf i,\quad h_2=\mathbf k,\quad h_3=-1,\quad h_4=-\mathbf i,\quad h_5=1.
$$
The following picture represents this closed chain in $SU(2)\cong S^3$ after a schematic planar projection. The arrows are not Euclidean side vectors: the label $g_j$ records the multiplicative increment $h_j=h_{j-1}g_j$.

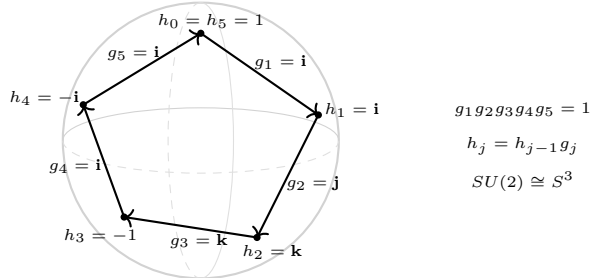
\begin{center}
\begin{tikzpicture}[scale=1.35, line cap=round, line join=round]


\coordinate (H0) at (0,1.05);
\coordinate (H1) at (1.15,0.25);
\coordinate (H2) at (0.55,-0.95);
\coordinate (H3) at (-0.75,-0.75);
\coordinate (H4) at (-1.15,0.35);

\draw[gray!35, thick] (0,0) circle (1.35);
\draw[gray!25, dashed] (-1.35,0) arc (180:360:1.35 and 0.32);
\draw[gray!35] (1.35,0) arc (0:180:1.35 and 0.32);
\draw[gray!25] (0,-1.35) arc (-90:90:0.32 and 1.35);
\draw[gray!25, dashed] (0,1.35) arc (90:270:0.32 and 1.35);

\draw[->, thick] (H0) -- (H1);
\draw[->, thick] (H1) -- (H2);
\draw[->, thick] (H2) -- (H3);
\draw[->, thick] (H3) -- (H4);
\draw[->, thick] (H4) -- (H0);

\fill[black] (H0) circle (1pt);
\fill[black] (H1) circle (1pt);
\fill[black] (H2) circle (1pt);
\fill[black] (H3) circle (1pt);
\fill[black] (H4) circle (1pt);

\node at (0.1,1.17) {\tiny{$h_0=h_5=1$}};
\node at (1.48,0.30) {\tiny{$h_1=\mathbf i$}};
\node at (0.70,-1.09) {\tiny{$h_2=\mathbf k$}};
\node at (-0.98,-0.93) {\tiny{$h_3=-1$}};
\node at (-1.53,0.42) {\tiny{$h_4=-\mathbf i$}};

\node at (0.78,0.77) {\tiny{$g_1=\mathbf i$}};
\node at (1.1,-0.42) {\tiny{$g_2=\mathbf j$}};
\node at (-0.02,-0.98) {\tiny{$g_3=\mathbf k$}};
\node at (-1.23,-0.25) {\tiny{$g_4=\mathbf i$}};
\node at (-0.66,0.86) {\tiny{$g_5=\mathbf i$}};

\node at (3.15,0.30) {\tiny{$g_1g_2g_3g_4g_5=1$}};
\node at (3.15,-0.05) {\tiny{$h_j=h_{j-1}g_j$}};
\node at (3.15,-0.40) {\tiny{$SU(2)\cong S^3$}};

\end{tikzpicture}
\end{center}

Thus this is a polygon in the Lie group $SU(2)$, not a Euclidean polygon. The prescribed conjugacy class plays the role of the side length, and the closing condition is multiplicative.
\end{Example}

\begin{Remark}\label{sympf}
Let us spell out the symplectic structures on $M_A$ and $\Mult_A$. On $M_A$ we will use the Kähler form coming from the natural ample line bundle in the construction of the moduli space of parabolic bundles. Equivalently, if one constructs $M_A$ as a GIT quotient, this is the form induced by the chosen linearization. On the stable locus this Kähler form is an honest symplectic form.

On $\Mult_A$ the construction is parallel to ordinary symplectic reduction, but the moment map takes values in the group $SU(2)$ rather than in the dual of its Lie algebra. Each conjugacy class $\mathcal C_i\subset SU(2)$ carries a canonical $SU(2)$-invariant two-form. On the product $\mathcal C_1\times\cdots\times\mathcal C_n$ these forms are combined in such a way that the map
$$
m:\mathcal C_1\times\cdots\times\mathcal C_n\longrightarrow SU(2),
\qquad
m(g_1,\ldots,g_n)=g_1\cdots g_n
$$
plays the role of the moment map. Restricting to $m^{-1}(1)$ and quotienting by simultaneous conjugation, this two-form descends to a closed non-degenerate two-form on the stable quotient
$$
m^{-1}(1)^{s}/SU(2)\subset \Mult_A.
$$
This is the natural symplectic form on the multiplicative polygon space.
\end{Remark}

\begin{Proposition}\label{prop:MS-multiplicative-new}
There is a natural homeomorphism $M_A\cong \Mult_A$. On the stable locus this homeomorphism is a real analytic diffeomorphism. If $A$ is generic, hence semistability equals stability, it is a real analytic diffeomorphism on the whole smooth moduli space.

Moreover,with these symplectic forms in Remark \ref{sympf}, the above identification is a symplectomorphism on the stable locus.
\end{Proposition}
\begin{proof}
Let $X^\circ=\PP^1\setminus\{p_1,\ldots,p_n\}$. Choose positively oriented simple loops $\gamma_i$ around $p_i$, based at the same point and ordered so that $
\gamma_1\cdots\gamma_n=1$. Then
$$
\pi_1(X^\circ)
=
\langle \gamma_1,\ldots,\gamma_n\mid \gamma_1\cdots\gamma_n=1\rangle.
$$
Therefore a homomorphism $\rho:\pi_1(X^\circ)\to SU(2)$ with $\rho(\gamma_i)\in\mathcal C_i$ is the same as a tuple $(g_1,\ldots,g_n)$ with $g_i\in\mathcal C_i$ and $g_1\cdots g_n=1$, by setting $g_i=\rho(\gamma_i)$. Changing the unitary frame at the base point conjugates all matrices $g_i$ by the same element of $SU(2)$. Hence the corresponding representation quotient is exactly
$$
\left\{
(g_1,\ldots,g_n)\in \mathcal C_1\times\cdots\times\mathcal C_n
\mid g_1\cdots g_n=1
\right\}/SU(2)
=
\Mult_A.
$$
We now compare this representation quotient with the parabolic moduli space. The stability condition used above can be written in the usual parabolic form by replacing the weights $0,a_i$ at $p_i$ with the shifted weights $-a_i/2,a_i/2$. Indeed, for a line subbundle $L\subset E$, the inequality $
\deg L+\sum_{L_{p_i}=V_i}a_i\leq \frac{1}{2}\sum_{i=1}^n a_i$ is equivalent to $
\deg L+\sum_{L_{p_i}=V_i}\frac{a_i}{2}
+\sum_{L_{p_i}\neq V_i}\left(-\frac{a_i}{2}\right)
\leq 0$. Thus the parabolic degree of $E$ is zero, since $\det E\cong\OO_{\PP^1}$ and the two shifted weights at each marked point have sum zero. For the shifted weights $-a_i/2,a_i/2$, the prescribed local unitary monodromy has eigenvalues $\exp(-\pi i a_i),\,\exp(\pi i a_i)$. Thus the local monodromy around $p_i$ lies in the conjugacy class $\mathcal C_i$.

By \cite[Theorem 4.1]{MehtaSeshadri1980}, polystable parabolic bundles of parabolic degree zero are identified with completely reducible unitary representations of $\pi_1(X^\circ)$ with the prescribed local conjugacy classes. In the stable case, this identifies stable parabolic bundles with irreducible unitary representations. Applied to the rank two determinant-one situation above, the theorem gives a natural bijection $M_A\longrightarrow \Mult_A$.

Concretely, a stable parabolic bundle determines its adapted flat unitary connection on $X^\circ$; the monodromy of this connection is the corresponding representation. Conversely, a unitary representation with local monodromies in the classes $\mathcal C_i$ determines a flat unitary bundle on $X^\circ$, and the asymptotic eigenspace data at each puncture determine the parabolic lines $V_i$ and the shifted weights $-a_i/2,a_i/2$.

This bijection is continuous, and both sides are compact Hausdorff spaces: $M_A$ is projective, while $\Mult_A$ is a quotient of the closed subset $g_1\cdots g_n=1$ of the compact space $\mathcal C_1\times\cdots\times\mathcal C_n$ by the compact group $SU(2)$. Hence the bijection is a homeomorphism.

On the stable locus the stabilizer is finite modulo the center, and both quotients are smooth real analytic orbifolds, smooth manifolds if the effective action is free. The construction above depends real analytically on the flat unitary connection and its local monodromy data; hence the homeomorphism restricts to a real analytic diffeomorphism on the stable locus. This is the sense in which the identification is real analytic: it is not a holomorphic isomorphism with respect to a complex structure on the representation side, but an isomorphism of the underlying real analytic spaces.

Finally, we compare the symplectic structures. On $M_A$ one has the natural Kähler form obtained from the parabolic GIT, construction of the moduli of parabolic bundles. On $\Mult_A$ each conjugacy class $\mathcal C_i$ carries its standard quasi-Hamiltonian $SU(2)$-structure, the product $\mathcal C_1\times\cdots\times\mathcal C_n$ has group-valued moment map $(g_1,\ldots,g_n)\longmapsto g_1\cdots g_n$, and $\Mult_A$ is the quasi-Hamiltonian reduction at $1$. The induced reduced two-form is the Goldman--Atiyah--Bott symplectic form for the character variety with fixed boundary holonomies. Under the correspondence above, the parabolic Kähler form and this reduced symplectic form agree, with the normalization fixed by the eigenvalues $\exp(\pm\pi i a_i)$ \cite{GuruprasadHuebschmannJeffreyWeinstein1997,AlekseevMalkinMeinrenken1998}. Therefore the real analytic diffeomorphism on the stable locus is a symplectomorphism.
\end{proof}

\subsection*{The multiplicative quiver viewpoint}

The multiplicative polygon space can also be encoded by a star-shaped quiver. Let $Q$ be the quiver with one central vertex $0$ and outer vertices $1,\ldots,n$. We attach the group $SU(2)$ to the central vertex, and we attach to the $i$-th leg the conjugacy class $\mathcal C_i\subset SU(2)$.

A compact multiplicative representation of this star-shaped quiver is a tuple $(g_1,\ldots,g_n)\in \mathcal C_1\times\cdots\times\mathcal C_n$. The multiplicative moment-map equation at the central vertex is the product relation $
g_1\cdots g_n=1$. The compact base-change group is $SU(2)$, acting by simultaneous conjugation: $
h\cdot(g_1,\ldots,g_n)=(hg_1h^{-1},\ldots,hg_nh^{-1})$. Thus the compact multiplicative quiver quotient is
$$
\mathcal M_Q^{\mathrm{mult},c}(A):=
\left\{
(g_1,\ldots,g_n)\in \mathcal C_1\times\cdots\times\mathcal C_n
\mid g_1\cdots g_n=1
\right\}/SU(2).
$$
This is the multiplicative analogue of the additive quiver construction. In the additive case the data lie in coadjoint orbits, the equation is the linear closing condition $u_1+\cdots+u_n=0$, and the quotient is taken by rotations. In the multiplicative case the data lie in conjugacy classes, the equation is the group-valued closing condition $g_1\cdots g_n=1$, and the quotient is taken by simultaneous conjugation.

\begin{Remark}
The construction above is the compact $SU(2)$ form of the multiplicative quiver viewpoint. The corresponding complex algebraic theory is given by multiplicative quiver varieties \cite{CrawleyBoeveyShaw2006,Yamakawa2008}.
\end{Remark}

\begin{Example}\label{Mult_Q_5}
For $n=5$, the compact multiplicative quiver has the same underlying star-shaped graph as the additive quiver in Example \ref{Add_Q_5}. What changes is not the graph, but the nature of the data attached to it.

In the additive case a representation is given by five vectors $q_1,\ldots,q_5\in\CC^2$, or, after quotienting by the scalar actions on the outer vertices, by five points $[q_1],\ldots,[q_5]\in\PP^1$. In the compact multiplicative case one instead fixes conjugacy classes $\mathcal C_i\subset SU(2)$ and considers tuples $
(g_1,\ldots,g_5)\in \mathcal C_1\times\cdots\times\mathcal C_5$.

The equation at the central vertex is no longer the additive moment-map equation, but the multiplicative relation $
g_1g_2g_3g_4g_5=1$. The compact base-change group is $SU(2)$, acting by simultaneous conjugation. Hence the corresponding compact multiplicative quiver quotient is
$$
\mathcal M_Q^{\operatorname{mult},c}(A)=
\left\{
(g_1,\ldots,g_5)\in \mathcal C_1\times\cdots\times\mathcal C_5
\mid
g_1g_2g_3g_4g_5=1
\right\}/SU(2).
$$
That is, if $\lambda_i=\exp(\pi i a_i)$, an element $g_i\in\mathcal C_i$ is determined by its $\lambda_i$-eigenline $L_i\in\PP^1$. Thus one may also view a point as a configuration $(L_1,\ldots,L_5)\in(\PP^1)^5$ satisfying the multiplicative closing condition $g_1(L_1)\cdots g_5(L_5)=1$, where $g_i(L_i)$ is the unique element of $SU(2)$ with eigenvalue $\lambda_i$ on $L_i$ and eigenvalue $\lambda_i^{-1}$ on $L_i^\perp$.

Thus the graph is the same as in Example \ref{Add_Q_5}, but the additive linear data $q_i:\CC\to\CC^2$ are replaced by unitary monodromy data $g_i\in\mathcal C_i$, and the additive closing equation is replaced by a product relation in $SU(2)$.
\end{Example}

\begin{Proposition}\label{prop:mult-quiver-polygon}
There is a natural identification $\mathcal M_Q^{\mathrm{mult},c}(A)=\Mult_A$.
\end{Proposition}

\begin{proof}
By construction, a compact multiplicative representation of the star-shaped quiver is a tuple
$$
(g_1,\ldots,g_n)\in \mathcal C_1\times\cdots\times\mathcal C_n.
$$
The compact base-change group is $SU(2)$ at the central vertex. It acts on such tuples by simultaneous conjugation: $
h\cdot(g_1,\ldots,g_n)
=
(hg_1h^{-1},\ldots,hg_nh^{-1}),
\, h\in SU(2)$.

The multiplicative moment-map equation at the central vertex is the product relation $g_1\cdots g_n=1$. Thus, if
$$
m:\mathcal C_1\times\cdots\times\mathcal C_n\longrightarrow SU(2),
\qquad
m(g_1,\ldots,g_n)=g_1\cdots g_n,
$$
then the space of solutions of the multiplicative quiver equation is
$$
m^{-1}(1)
=
\left\{
(g_1,\ldots,g_n)\in \mathcal C_1\times\cdots\times\mathcal C_n
\mid g_1\cdots g_n=1
\right\}.
$$
The compact multiplicative quiver quotient is therefore
$$
\mathcal M_Q^{\mathrm{mult},c}(A)
=
m^{-1}(1)/SU(2),
$$
where $SU(2)$ acts as above. On the other hand, the multiplicative polygon space was defined as
$$
\Mult_A
=
\left\{
(g_1,\ldots,g_n)\in \mathcal C_1\times\cdots\times\mathcal C_n
\mid g_1\cdots g_n=1
\right\}/SU(2),
$$
again with respect to simultaneous conjugation. Hence the set of solutions and the equivalence relation are exactly the same on both sides. Therefore the identity map on $m^{-1}(1)$ descends to a canonical bijection
$$
\mathcal M_Q^{\mathrm{mult},c}(A)\longrightarrow \Mult_A.
$$
Since both quotients are endowed with the same quotient topology, this bijection is the natural identification claimed.
\end{proof}

\begin{Remark}
There is a configuration-theoretic way to describe $\Mult_A$, but it is different from the GIT configuration quotient of the additive case. Fix the standard Hermitian form on $\CC^2$, and set
$
\lambda_j=\exp(\pi i a_j)
$
for $j=1,\ldots,n$. For a line $L\in\PP^1=\PP(\CC^2)$, let $L^\perp$ be its Hermitian orthogonal complement. There is a unique element $g_j(L)\in SU(2)$ whose eigenvalue on $L$ is $\lambda_j$ and whose eigenvalue on $L^\perp$ is $\lambda_j^{-1}$. Equivalently,
$$
g_j(L)=\lambda_j P_L+\lambda_j^{-1}(\operatorname{Id}-P_L),
$$
where $P_L$ is the Hermitian orthogonal projection onto $L$. We define the space of multiplicative configurations of weight $A$ by
$$
\operatorname{Conf}^{\operatorname{mult}}_A(\PP^1):=
\left\{
(L_1,\ldots,L_n)\in(\PP^1)^n
\mid
g_1(L_1)\cdots g_n(L_n)=1
\right\}.
$$
The group $SU(2)$ acts diagonally on $(\PP^1)^n$, and the construction of $g_j(L_j)$ is $SU(2)$-equivariant. Hence $SU(2)$ preserves $\operatorname{Conf}^{\operatorname{mult}}_A(\PP^1)$, and we obtain
$$
\Mult_A
\cong
\operatorname{Conf}^{\operatorname{mult}}_A(\PP^1)/SU(2).
$$
Thus a multiplicative polygon may be viewed as an ordered configuration of points on $\PP^1$ satisfying a multiplicative closing condition.

This should not be confused with the additive GIT quotient of configurations. In the additive case the equation $u_1+\cdots+u_n=0$ is the ordinary moment-map equation for the Hamiltonian action of $SU(2)$ on $(\PP^1)^n$, and Kempf--Ness identifies the quotient with $(\PP^1)^n\quot\PGL_2$. In the multiplicative case the condition
$
g_1(L_1)\cdots g_n(L_n)=1
$
is a monodromy relation in $SU(2)$. Moreover, the construction of $g_j(L_j)$ uses the Hermitian complement $L_j^\perp$, so it is natural for the unitary group $SU(2)$, not for the full complex group $\PGL_2$. Therefore $\operatorname{Conf}^{\operatorname{mult}}_A(\PP^1)/SU(2)$ is a unitary multiplicative configuration quotient, not a Hassett--GIT quotient or a moduli space of weighted pointed rational curves.
\end{Remark}

\begin{Remark}
In the multiplicative dictionary the weights are not defined only up to scale. A vector $A=(a_1,\ldots,a_n)\in\Delta$, with $0<a_i<1$, is used simultaneously as the parabolic weight vector, as the angle vector of the prescribed conjugacy classes, and as the parameter of the compact multiplicative quiver quotient. More precisely, set $\lambda_i=\exp(\pi i a_i)$ and let $\mathcal C_i\subset SU(2)$ be the conjugacy class with eigenvalues $\lambda_i$ and $\lambda_i^{-1}$. The moduli space $M_A$ uses the parabolic weights $A$, while $\Mult_A$ and $\mathcal M_Q^{\operatorname{mult},c}(A)$ use the conjugacy classes $\mathcal C_1,\ldots,\mathcal C_n$ and the multiplicative relation $g_1\cdots g_n=1$. The multiplicative configuration model uses the same data in eigenline form: for a line $L\in\PP^1$, the element attached to $L$ is
$$
g_i(L)=\lambda_i P_L+\lambda_i^{-1}(\operatorname{Id}-P_L),
$$
where $P_L$ is the Hermitian projection onto $L$. Thus the passage among the parameters is
$$
A=(a_i)
\quad\longleftrightarrow\quad
(\lambda_i=\exp(\pi i a_i))
\quad\longleftrightarrow\quad
(\mathcal C_i)
\quad\longleftrightarrow\quad
\operatorname{Conf}^{\operatorname{mult}}_A(\PP^1).
$$
Conversely, a non-central conjugacy class $\mathcal C_i\subset SU(2)$ determines $a_i\in(0,1)$ uniquely by $\operatorname{tr}(g)=2\cos(\pi a_i)$ for $g\in\mathcal C_i$. Hence, unlike in the additive case, replacing $A$ by $\lambda A$ is not merely a normalization: it changes the conjugacy classes, the monodromy equation, the parabolic weights, and therefore the corresponding moduli spaces. Wall-crossing in the multiplicative setting means moving the actual point $A\in\Delta$ across the Bauer--Mukai walls; the same variation changes at once the parabolic stability chamber, the prescribed unitary conjugacy classes, and the compact multiplicative quiver parameters.
\end{Remark}

We summarize the multiplicative dictionary in the following statement.

\begin{thm}\label{thm:multiplicative-summary}
Let $A=(a_1,\ldots,a_n)\in\Delta$ be generic, with $0<a_i<1$ for all $i$. Let $C_A$ be the chamber of the Bauer--Mukai decomposition of $\Delta$ containing $A$, and let $\sigma_A\subset\Eff(Z_n)$ be the corresponding Mori chamber. If $Z_A$ is the Mori model of $Z_n$ associated with $\sigma_A$, then there are natural identifications
$$
Z_A
\cong
M_A
\cong
\Mult_A
\cong
\operatorname{Conf}^{\operatorname{mult}}_A(\PP^1)/SU(2)
=
\mathcal M_Q^{\operatorname{mult},c}(A).
$$
More precisely, $Z_A\cong M_A$ is an isomorphism of projective varieties. The identification $M_A\cong\Mult_A$ is the Mehta--Seshadri correspondence: it is a homeomorphism, and on the stable locus it is a real analytic symplectomorphism. 
\end{thm}

\begin{proof}
The chamber $\sigma_A$ is, by construction, the cone over the chamber $C_A$ of the demi-hypercube containing the parabolic weight $A$. By Theorem \ref{thm:BM-model}, the Mori model of $Z_n$ corresponding to $\sigma_A$ is the parabolic moduli space $M_A$. This yields the algebraic isomorphism $Z_A\cong M_A$.

By Proposition \ref{prop:MS-multiplicative-new}, the moduli space $M_A$ is identified with the unitary character variety
$$
\Mult_A=
\left\{
(g_1,\ldots,g_n)\in\mathcal C_1\times\cdots\times\mathcal C_n
\mid
g_1\cdots g_n=1
\right\}/SU(2).
$$
This is the Mehta--Seshadri identification. It is a homeomorphism on the compact moduli spaces, and on the stable locus it is a real analytic diffeomorphism. With the Kähler form on $M_A$ and the reduced symplectic form on $\Mult_A$, it is a symplectomorphism on the stable locus.

We now describe the configuration-theoretic form of $\Mult_A$. Fix the standard Hermitian form on $\CC^2$ and set $\lambda_j=\exp(\pi i a_j)$. For a line $L\in\PP^1=\PP(\CC^2)$, let $L^\perp$ be its Hermitian orthogonal complement. There is a unique element $g_j(L)\in SU(2)$ whose eigenvalue on $L$ is $\lambda_j$ and whose eigenvalue on $L^\perp$ is $\lambda_j^{-1}$. Explicitly,
$$
g_j(L)=\lambda_jP_L+\lambda_j^{-1}(\operatorname{Id}-P_L),
$$
where $P_L$ is the Hermitian orthogonal projection onto $L$. This yields an $SU(2)$-equivariant identification $\PP^1\cong\mathcal C_j$, sending $L$ to $g_j(L)$.

Define the space of multiplicative configurations of weight $A$ by
$$
\operatorname{Conf}^{\operatorname{mult}}_A(\PP^1):=
\left\{
(L_1,\ldots,L_n)\in(\PP^1)^n
\mid
g_1(L_1)\cdots g_n(L_n)=1
\right\}.
$$
The diagonal action of $SU(2)$ on $(\PP^1)^n$ preserves this locus, since the construction $L\mapsto g_j(L)$ is $SU(2)$-equivariant. Therefore the eigenline map induces a natural identification
$$
\operatorname{Conf}^{\operatorname{mult}}_A(\PP^1)/SU(2)
\cong
\Mult_A.
$$
This is a configuration-space description, but not a GIT quotient by $\PGL_2$: the construction uses the Hermitian complement $L^\perp$, and the defining equation is the monodromy relation $g_1(L_1)\cdots g_n(L_n)=1$.

Finally, by Proposition \ref{prop:mult-quiver-polygon}, the compact multiplicative quiver quotient of the star-shaped quiver is
$$
\mathcal M_Q^{\operatorname{mult},c}(A)
=
\left\{
(g_1,\ldots,g_n)\in\mathcal C_1\times\cdots\times\mathcal C_n
\mid
g_1\cdots g_n=1
\right\}/SU(2).
$$
This is exactly $\Mult_A$. Hence
$$
Z_A
\cong
M_A
\cong
\Mult_A
\cong
\operatorname{Conf}^{\operatorname{mult}}_A(\PP^1)/SU(2)
=
\mathcal M_Q^{\operatorname{mult},c}(A),
$$
with the stated nature of each identification.
\end{proof}

\section{Automorphisms of polygon and quiver moduli spaces}\label{Sec4}

We now discuss automorphisms of the polygon and quiver moduli spaces appearing above. The additive and multiplicative pictures have different algebraic natures.

In the additive case, the algebraic structure is natural from the outset. Indeed, the coadjoint orbit $S^2_{b_i}$ is identified with the polarized projective line $(\PP^1,\OO_{\PP^1}(b_i))$, and the Kempf--Ness theorem identifies the symplectic quotient defining $\Padd_{\bb}$ with the GIT quotient $(\PP^1)^n\quot_{\bb}\PGL_2$. Hence $\Padd_{\bb}$ carries the natural projective structure coming from this GIT quotient. The same holds for the ordinary additive quiver moduli space, since $\mathcal M_Q(\bb)\cong(\PP^1)^n\quot_{\bb}\PGL_2\cong\Padd_{\bb}$.

In the multiplicative case, the quotient $\Mult_A$ is naturally a compact real analytic symplectic quotient. It can also be viewed as a real algebraic quotient in its unitary presentation, since $SU(2)$ and the conjugacy classes $\mathcal C_i$ are real algebraic. However, the complex projective structure used in this section is the one obtained from the parabolic moduli space through the identifications of Theorem \ref{thm:multiplicative-summary}. Thus we endow $
\Mult_A,\,
\operatorname{Conf}^{\operatorname{mult}}_A(\PP^1)/SU(2),
\,
\mathcal M_Q^{\operatorname{mult},c}(A)$ with the unique projective algebraic structures for which the identifications $
M_A
\cong
\Mult_A
\cong
\operatorname{Conf}^{\operatorname{mult}}_A(\PP^1)/SU(2)
=
\mathcal M_Q^{\operatorname{mult},c}(A)$ are isomorphisms of projective varieties. In what follows, $\Aut(\Mult_A)$, $\Aut(\operatorname{Conf}^{\operatorname{mult}}_A(\PP^1)/SU(2))$, $\Aut(\mathcal M_Q^{\operatorname{mult},c}(A))$ always refer to this projective algebraic structure.

\subsection*{Automorphisms in the multiplicative case}

Let $p_1,\ldots,p_n\in\PP^1$ be general points, and let $A=(a_1,\ldots,a_n)$ be a parabolic weight. Recall that $M_A$ denotes the moduli space of rank two parabolic bundles with trivial determinant on $\PP^1$, semistable with respect to $A$.

Let $R\subset\{1,\ldots,n\}$ have even cardinality. Given a quasi-parabolic bundle $(E,V_1,\ldots,V_n)$, the elementary transformation centered at $R$ is obtained from the exact sequence
$$
0\longrightarrow E'\longrightarrow E\longrightarrow
\bigoplus_{i\in R}(E_{p_i}/V_i)\otimes\OO_{p_i}
\longrightarrow 0
$$
and then twisting $E'$ by $\OO_{\PP^1}(\sharp R/2)$ in order to restore degree zero. We will denote the resulting transformation by $\operatorname{el}_R$. The transformations $\operatorname{el}_R$, with $\sharp R$ even, form a group $El\cong(\ZZ/2\ZZ)^{n-1}$.

The transformation $\operatorname{el}_R$ changes the weights by replacing $a_i$ with $1-a_i$ for $i\in R$. Thus we set $A^R=(a'_1,\ldots,a'_n)$, where
$$
a'_i=
\begin{cases}
1-a_i & \text{if } i\in R,\\
a_i & \text{if } i\notin R.
\end{cases}
$$
Then $\operatorname{el}_R$ gives an isomorphism $M_A\to M_{A^R}$. Let $C_A$ be the chamber of the Bauer wall-and-chamber decomposition containing $A$, that is the set of weights defining the same stability condition as $A$. We say that $\operatorname{el}_R$ is $A$-admissible if $A^R\in C_A$, and we set
$$
El_A:=\{\operatorname{el}_R\in El\mid A^R\in C_A\}.
$$
If $\operatorname{el}_R\in El_A$, then $A$ and $A^R$ define the same stability condition, so $\operatorname{el}_R$ induces an automorphism of $M_A$.

The main results of \cite{AraujoFassarellaKaurMassarenti2019} say the following. For the central weight $A_F=(1/2,\ldots,1/2)$, one has $\Aut(M_{A_F})=El\cong(\ZZ/2\ZZ)^{n-1}$. More generally, for weights $A$ in the interior of the chopped demi-hypercube of \cite[Corollary 1.4]{AraujoFassarellaKaurMassarenti2019}, that is for the weights for which $M_A$ is a small modification of the Fano model $M_{A_F}$, one has $\Aut(M_A)=El_A$.

The projective structures fixed above allow us to transfer the algebraic automorphism results for $M_A$ to all multiplicative incarnations of the same moduli problem.

\begin{thm}\label{thm:aut-central-multiplicative-polygons}
Let $n\geq 5$, and let $A_F=(1/2,\ldots,1/2)$. Then
$$
\Aut(\Mult_{A_F})
\cong
\Aut\left(\operatorname{Conf}^{\operatorname{mult}}_{A_F}(\PP^1)/SU(2)\right)
\cong
\Aut\left(\mathcal M_Q^{\operatorname{mult},c}(A_F)\right)
\cong
(\ZZ/2\ZZ)^{n-1}.
$$
This group is generated by the automorphisms induced by the elementary transformations $\operatorname{el}_R$, with $R\subset\{1,\ldots,n\}$ of even cardinality.
\end{thm}

\begin{proof}
By \cite[Theorem 1.2]{AraujoFassarellaKaurMassarenti2019}, $\Aut(M_{A_F})=El\cong(\ZZ/2\ZZ)^{n-1}$. By Theorem \ref{thm:multiplicative-summary}, the spaces $\Mult_{A_F}$, $\operatorname{Conf}^{\operatorname{mult}}_{A_F}(\PP^1)/SU(2)$ and $\mathcal M_Q^{\operatorname{mult},c}(A_F)$ carry the projective structures transported from $M_{A_F}$. With these structures they are isomorphic to $M_{A_F}$ as projective varieties. Hence their algebraic automorphism groups are all identified with $\Aut(M_{A_F})=El$.
\end{proof}

\begin{thm}\label{thm:aut-general-multiplicative-polygons}
Let $A$ be a generic weight in the interior of the chopped demi-hypercube of \cite[Corollary 1.4]{AraujoFassarellaKaurMassarenti2019}. Then
$$
\Aut(\Mult_A)
\cong
\Aut\left(\operatorname{Conf}^{\operatorname{mult}}_A(\PP^1)/SU(2)\right)
\cong
\Aut\left(\mathcal M_Q^{\operatorname{mult},c}(A)\right)
\cong
El_A.
$$
\end{thm}

\begin{proof}
By \cite[Corollary 1.4]{AraujoFassarellaKaurMassarenti2019}, the algebraic automorphism group of $M_A$ is $El_A$. The projective structures on $\Mult_A$, $\operatorname{Conf}^{\operatorname{mult}}_A(\PP^1)/SU(2)$ and $\mathcal M_Q^{\operatorname{mult},c}(A)$ are defined so that the identifications with $M_A$ in Theorem \ref{thm:multiplicative-summary} are algebraic isomorphisms. Therefore all three algebraic automorphism groups are identified with $\Aut(M_A)=El_A$.
\end{proof}

Let us spell out the modular meaning of the automorphisms above. On the parabolic side, $\operatorname{el}_R$ modifies a parabolic bundle by an elementary transformation centered at the parabolic points indexed by $R$. It replaces the bundle by the kernel of the quotient supported at those points, changes the parabolic directions accordingly, and then twists by $\OO_{\PP^1}(\sharp R/2)$. The weights are changed by $A\mapsto A^R$. If $R$ is $A$-admissible, then $A$ and $A^R$ lie in the same chamber, so the stability condition is unchanged and $\operatorname{el}_R$ becomes an automorphism of $M_A$.

On the multiplicative polygon side, a point is a closed tuple $(g_1,\ldots,g_n)$ with $g_i\in\mathcal C_i$ and $g_1\cdots g_n=1$, modulo simultaneous conjugation. The automorphism induced by $\operatorname{el}_R$ is obtained as follows: take the parabolic bundle corresponding to $(g_1,\ldots,g_n)$ by the Mehta--Seshadri correspondence, apply $\operatorname{el}_R$, and then pass back to the unitary monodromy representation. Thus, for a general weight, the automorphism is not simply a pointwise operation on the matrices $g_i$.

For the central weight $A_F=(1/2,\ldots,1/2)$, the monodromy description becomes more concrete. If $R$ has even cardinality, then
$$
(g_1,\ldots,g_n)\longmapsto (g'_1,\ldots,g'_n),
\qquad
g'_i=
\begin{cases}
-g_i & \text{if } i\in R,\\
g_i & \text{if } i\notin R
\end{cases}
$$
preserves the product relation since $\sharp R$ is even. It also preserves the conjugacy classes, since for $a_i=1/2$ the eigenvalues are $i$ and $-i$.

On the multiplicative configuration side, write $\lambda_j=\exp(\pi i a_j)$ and let $g_j(L_j)$ be the element of $SU(2)$ with eigenvalue $\lambda_j$ on $L_j$ and $\lambda_j^{-1}$ on $L_j^\perp$. A point of $\operatorname{Conf}^{\operatorname{mult}}_A(\PP^1)/SU(2)$ is represented by a configuration $(L_1,\ldots,L_n)$ satisfying $g_1(L_1)\cdots g_n(L_n)=1$. The automorphism induced by $\operatorname{el}_R$ is obtained by translating this configuration into the monodromy tuple, applying the parabolic elementary transformation through Mehta--Seshadri, and translating back to eigenlines. In the central case this is simply
$$
L_i\longmapsto
\begin{cases}
L_i^\perp & \text{if } i\in R,\\
L_i & \text{if } i\notin R,
\end{cases}
$$
since $-g_i(L_i)=g_i(L_i^\perp)$ when $a_i=1/2$.

On the compact multiplicative quiver side, a point of $\mathcal M_Q^{\operatorname{mult},c}(A)$ is the same closed tuple of conjugacy-class elements, interpreted as a compact multiplicative representation of the star-shaped quiver. Thus the automorphisms above act on the quiver quotient by the same monodromy procedure: pass from the quiver point to the corresponding parabolic bundle, apply the elementary transformation, and return to the compact multiplicative quiver quotient. In the central case this is the sign change on the legs indexed by $R$.

\subsection*{Automorphisms in the additive case}

We now discuss the additive analogue of the automorphism results obtained in the multiplicative case. Let
$$
Y_{n-3}:=\Bl_{q_1,\ldots,q_{n-1}}\PP^{n-3},
$$
where the points $q_i$ are general. We only consider chambers contained in $\Mov(Y_{n-3})$, that is chambers corresponding to small $\QQ$-factorial modifications of $Y_{n-3}$.

Let $\mathcal C\subset\Mov(Y_{n-3})$ be such a chamber, and let $\phi_{\mathcal C}:Y_{n-3}\dashrightarrow Y_{\mathcal C}$ be the corresponding small $\QQ$-factorial modification. Through the additive dictionary, $Y_{\mathcal C}$ is identified with a GIT quotient of configurations of $n$ ordered points on $\PP^1$, with a Hassett--GIT moduli space, with the fixed-trivial-bundle parabolic quotient, with the ordinary additive quiver moduli space, and with the additive polygon space.

We first describe the relevant subgroup of $S_n$ intrinsically in terms of the chamber. Let $A=(a_1,\ldots,a_n)$ be a weight vector in the relative interior of the GIT chamber corresponding to $\mathcal C$, normalized by $a_1+\cdots+a_n=2$. The chamber is cut out by the signs of the inequalities $\sum_{i\in I}a_i<1,\, I\subset\{1,\ldots,n\}$. Thus we associate to $\mathcal C$ the simplicial complex
$$
\mathcal K_{\mathcal C}:=
\left\{
I\subset\{1,\ldots,n\}
\mid
\sum_{i\in I}a_i<1
\right\}.
$$
This is independent of the choice of $A$ in the chamber. Its subsets are exactly the collections of markings that are allowed to collide in the corresponding Hassett--GIT moduli problem. Define
$$
\Gamma_{\mathcal C}:=\Aut(\mathcal K_{\mathcal C})
=
\left\{
\sigma\in S_n
\mid
I\in\mathcal K_{\mathcal C}
\Longleftrightarrow
\sigma(I)\in\mathcal K_{\mathcal C}
\right\}.
$$
That is
$$
\Gamma_{\mathcal C}
=
\left\{
\sigma\in S_n
\mid
\sum_{i\in I}a_i<1
\Longleftrightarrow
\sum_{i\in\sigma(I)}a_i<1
\text{ for every }I\subset\{1,\ldots,n\}
\right\}.
$$
In particular, $\Gamma_{\mathcal C}$ contains the product of symmetric groups permuting equal weights, but it may be larger if the threshold complex $\mathcal K_{\mathcal C}$ has additional symmetries.

\begin{thm}\label{thm:additive-small-aut}
Let $\mathcal C\subset\Mov(Y_{n-3})$ be a maximal chamber of the Mori chamber decomposition, and let $Y_{\mathcal C}$ be the corresponding small $\QQ$-factorial modification of $Y_{n-3}$. Then $\Aut(Y_{\mathcal C})\cong \Gamma_{\mathcal C}$. Consequently, for any additive moduli realization corresponding to $\mathcal C$ one has
$$
\Aut(Y_{\mathcal C})
\cong
\Aut((\PP^1)^n_{\bb})
\cong
\Aut(\cM_{0,\avec}^{\mathrm{GIT}})
\cong
\Aut(\Mtriv_{\bb})
\cong
\Aut(\mathcal M_Q(\bb))
\cong
\Aut(\Padd_{\bb})
\cong
\Gamma_{\mathcal C}.
$$
\end{thm}

\begin{proof}
Let $\Sigma_{n-3}$ be the symmetric GIT quotient $\Sigma_{n-3}:=(\PP^1)^n\quot\PGL_2$. This is the Fano chamber model in the additive Mori chamber decomposition. By \cite[Theorems 4.4 and 4.7]{BolognesiMassarenti2021}, $
\Aut(\Sigma_{n-3})\cong S_n$. Moreover, for the group of pseudo-automorphisms we have $\PsAut(\Sigma_{n-3})=\Aut(\Sigma_{n-3})$. Indeed, if $n$ is even then $\Sigma_{n-3}$ has Picard rank one, while if $n$ is odd then $\Sigma_{n-3}$ is smooth Fano; in both cases this follows from \cite[Proposition 7.2]{Massarenti2020PLMS}.

Let $\phi_F:Y_{n-3}\dashrightarrow \Sigma_{n-3}$ be the birational map corresponding to the Fano chamber. It is an isomorphism in codimension one. Hence conjugation by $\phi_F$ gives
$$
\PsAut(Y_{n-3})
\cong
\PsAut(\Sigma_{n-3})
=
\Aut(\Sigma_{n-3})
\cong
S_n.
$$
Under this identification, the group $S_n$ is generated by the obvious permutations of the $n-1$ points $q_1,\ldots,q_{n-1}$ together with the standard Cremona transformations. Now let $\phi_{\mathcal C}:Y_{n-3}\dashrightarrow Y_{\mathcal C}$ be the small modification associated with $\mathcal C$. Any automorphism of $Y_{\mathcal C}$ induces a pseudo-automorphism of $Y_{n-3}$ by conjugation: $
\psi
\mapsto
\phi_{\mathcal C}^{-1}\circ\psi\circ\phi_{\mathcal C}$. Thus $\Aut(Y_{\mathcal C})$ is a subgroup of $\PsAut(Y_{n-3})\cong S_n$.

A pseudo-automorphism $\sigma\in S_n$ induces a biregular automorphism of $Y_{\mathcal C}$ if and only if it preserves the nef cone of $Y_{\mathcal C}$, pulled back to $\Mov(Y_{n-3})$. This pulled-back nef cone is precisely the closure of $\mathcal C$. Hence
$$
\Aut(Y_{\mathcal C})
=
\{\sigma\in S_n\mid \sigma(\mathcal C)=\mathcal C\}.
$$
It remains only to rewrite this stabilizer in modular terms. The map from divisor classes on $Y_{n-3}$ to weights sends the Mori chamber decomposition of $\Mov(Y_{n-3})$ to the VGIT chamber decomposition for configurations of points on $\PP^1$. The walls are the hyperplanes $
\sum_{i\in I}a_i=1$. Therefore a permutation $\sigma\in S_n$ preserves $\mathcal C$ if and only if it preserves the sign pattern of all inequalities $\sum_{i\in I}a_i<1$. This is exactly the condition $
\sigma\in\Aut(\mathcal K_{\mathcal C})=\Gamma_{\mathcal C}$. This proves $
\Aut(Y_{\mathcal C})\cong\Gamma_{\mathcal C}$.

The remaining identifications follow from the additive dictionary: for weights in the chamber corresponding to $\mathcal C$, the spaces $
Y_{\mathcal C},\,
(\PP^1)^n_{\bb},\,
\cM_{0,\avec}^{\mathrm{GIT}},\,
\Mtriv_{\bb},\,
\mathcal M_Q(\bb),\,
\Padd_{\bb}
$ are naturally identified as projective varieties, the last one by Kempf--Ness. Hence their algebraic automorphism groups are all identified with $\Aut(Y_{\mathcal C})$.
\end{proof}

\begin{Remark}
The description by $\mathcal K_{\mathcal C}$ records exactly which subsets of markings may collide. Thus $\Gamma_{\mathcal C}$ is the group of relabellings of the markings preserving the collision combinatorics of the moduli problem.
\end{Remark}

Let $\sigma\in\Gamma_{\mathcal C}$. On the configuration quotient $(\PP^1)^n_{\bb}$, the automorphism is induced by relabelling the marked points: $
(x_1,\ldots,x_n)
\mapsto
(x_{\sigma^{-1}(1)},\ldots,x_{\sigma^{-1}(n)})$, followed by the canonical identification of the GIT quotients attached to two weights in the same chamber. If $\sigma$ actually preserves the weight vector $\bb$, this is literally a relabelling of configurations with the same linearization.

On the Hassett--GIT space $\cM_{0,\avec}^{\mathrm{GIT}}$, the same automorphism sends $
(C,x_1,\ldots,x_n)
\mapsto
(C,x_{\sigma^{-1}(1)},\ldots,x_{\sigma^{-1}(n)})$. The condition $\sigma\in\Gamma_{\mathcal C}$ says precisely that this relabelling preserves the stability rule, that is the subsets of markings that are allowed to collide.

On the fixed-trivial-bundle parabolic quotient $\Mtriv_{\bb}$, a point is represented by the trivial bundle $\OO_{\PP^1}^{\oplus 2}$ together with parabolic directions at the marked points. The automorphism relabels the marked points and their parabolic lines. Again, if the weights are not literally fixed, the relabelling is followed by the canonical chamber identification.

On the additive quiver side, $\mathcal M_Q(\bb)$ is the GIT quotient of the star-shaped quiver with one central vertex and $n$ legs. The automorphism induced by $\sigma$ permutes the legs of the quiver. The condition $\sigma\in\Gamma_{\mathcal C}$ says that this permutation preserves the chamber of the stability parameter. 

On the additive polygon side, $\Padd_{\bb}$ parametrizes closed polygons $
u_1+\cdots+u_n=0,\, |u_i|=b_i$, modulo rotations. If $\sigma$ preserves the length vector, it acts by permuting the sides: $
(u_1,\ldots,u_n)
\mapsto
(u_{\sigma^{-1}(1)},\ldots,u_{\sigma^{-1}(n)})$. For a general $\sigma\in\Gamma_{\mathcal C}$ the side lengths are not necessarily fixed pointwise; in that case the modular meaning is relabelling of the corresponding configuration or quiver data, followed by the canonical identification of the models inside the same VGIT chamber.

\subsection*{A non-small example}

The restriction to small modifications is essential. Consider a divisorial chamber whose model is $
Y':=\Bl_{q_1,\ldots,q_{n-2}}\PP^{n-3}$. This is obtained from $Y_{n-3}=\Bl_{q_1,\ldots,q_{n-1}}\PP^{n-3}$ by contracting one exceptional divisor, say $E_{n-1}$. Hence the map $
Y_{n-3}\rightarrow Y'$ is not small.

In the weight picture this corresponds to crossing a boundary wall of the movable cone. For instance, with distinguished label $n$, the condition $
a_{n-1}+a_n=1$ is the wall on which the divisor corresponding to the collision of the markings $n-1$ and $n$ is contracted. In the Hassett--GIT interpretation, this means that the boundary divisor where these two markings collide becomes strictly semistable and is collapsed in the quotient.

The variety $Y'$ is toric. Indeed, after a projective change of coordinates, the $n-2$ points may be taken to be the coordinate points of $\PP^{n-3}$. Thus the blow-up is a toric blow-up. Consequently $\Aut(Y')$ is no longer finite. At least the dense torus $
(\CC^*)^{n-3} $ acts on $Y'$, and the finite label symmetries of the blown-up coordinate points act as well. More precisely, for $n\geq 5$ the toric automorphism group contains $
(\CC^*)^{n-3}\rtimes S_{n-2}$, and the full toric symmetry also includes the fan symmetries of the corresponding toric fan.

Modularly, these additional automorphisms are not only relabellings of the marked points. On the dense configuration locus, one may normalize three points on $\PP^1$ and use cross-ratio-type coordinates for the remaining points. The torus acts by rescaling these coordinates. Thus, in the non-small chamber, the quotient has forgotten enough boundary information that new positive-dimensional automorphisms appear. This is exactly why the theorem above is formulated only for small modifications: in the small chamber case the automorphism group is the finite chamber symmetry group $\Gamma_{\mathcal C}$, while after divisorial contractions new automorphisms may appear.

In terms of additive polygons, the additional automorphisms of the non-small toric model are not relabellings of the sides. On a dense open subset one can describe a polygon by choosing a triangulation, that is by adding the diagonals
$
d_k=u_1+\cdots+u_{k+1}.
$
Once the lengths of the sides and of these diagonals are fixed, the remaining parameters are the angles describing how the triangles are glued along the diagonals.

The compact torus $(S^1)^{n-3}$ changes these angles. Geometrically, this means that one rotates one part of the polygon around a chosen diagonal, while keeping all side lengths and the closing condition
$
u_1+\cdots+u_n=0
$
unchanged. In the toric model this action extends to the algebraic torus $(\CC^*)^{n-3}$. The compact part changes the gluing angles of the polygon, while the non-compact part is more naturally seen in toric, or cross-ratio, coordinates, where it rescales the corresponding coordinates.

\begin{Example}\label{ex:aut-Ynminus3}
Consider the chamber corresponding to the original model $
Y_{n-3}=\Bl_{q_1,\ldots,q_{n-1}}\PP^{n-3}$. We use the Kapranov--Kumar labelling in which the points $q_1,\ldots,q_{n-1}$ correspond to the labels $1,\ldots,n-1$, while the label $n$ is the distinguished label. In the polygon interpretation, the labels $1,\ldots,n$ are the labels of the sides $
u_1,\ldots,u_n,
\,
u_1+\cdots+u_n=0,
\,
|u_i|=b_i$. Thus the model $Y_{n-3}$ singles out the side labelled by $n$.

Let $A=(a_1,\ldots,a_n)$ be a normalized weight vector in the chamber corresponding to $Y_{n-3}$, so that $\sum_i a_i=2$. In divisor coordinates this chamber is the nef chamber of $Y_{n-3}$. Under the weight map it is described by $
a_i+a_n<1 \text{ for all }i=1,\ldots,n-1$, and $
a_i+a_j+a_n>1 \text{ for all }1\leq i<j\leq n-1$.
Therefore the associated collision complex is
$$
\mathcal K_Y=
\left\{
I\subset\{1,\ldots,n\}
\mid
\sum_{i\in I}a_i<1
\right\}.
$$
More explicitly, $
I\in\mathcal K_Y$ if and only if either $n\notin I$ and $\sharp I\leq n-3$, or $n\in I$ and $\sharp I\leq 2$. Hence, for $n\geq 6$, the vertex $n$ is characterized intrinsically by the fact that no face containing $n$ has cardinality at least $3$. Thus every automorphism of $\mathcal K_Y$ fixes $n$, and we obtain $\Aut(\mathcal K_Y)\cong S_{n-1}$. By Theorem \ref{thm:additive-small-aut},
$$
\Aut(Y_{n-3})\cong S_{n-1}
\qquad
\text{for }n\geq 6.
$$
Modularly, this group permutes the first $n-1$ labels and fixes the distinguished label $n$. On the configuration quotient it acts by $
(x_1,\ldots,x_{n-1},x_n)
\mapsto
(x_{\sigma^{-1}(1)},\ldots,x_{\sigma^{-1}(n-1)},x_n),
\,
\sigma\in S_{n-1}$. On the Hassett--GIT space it relabels the first $n-1$ markings and keeps the distinguished marking fixed. On the additive quiver moduli space it permutes the first $n-1$ legs of the star-shaped quiver and fixes the leg labelled by $n$.

On the additive polygon space, the same operation relabels the first $n-1$ sides and keeps the side $u_n$ fixed:
$$
(u_1,\ldots,u_{n-1},u_n)
\longmapsto
(u_{\sigma^{-1}(1)},\ldots,u_{\sigma^{-1}(n-1)},u_n).
$$
If the side lengths satisfy $b_{\sigma(i)}=b_i$ for all $i$, this is literally an automorphism of the polygon space with fixed length vector $\bb$. For a general length vector in the same chamber, it is understood as the relabelling isomorphism from $\Padd_{\bb}$ to $\Padd_{\sigma\bb}$ followed by the canonical identification of the two GIT models, since $\bb$ and $\sigma\bb$ lie in the same chamber.

When $n=5$, the situation is exceptional. In this case $
Y_2=\Bl_{q_1,\ldots,q_4}\PP^2$ is the del Pezzo surface of degree $5$. The collision complex is the complete graph on the five labels, so its automorphism group is $S_5$, not only the subgroup fixing the distinguished label. Accordingly, $
\Aut(Y_2)\cong S_5$. Geometrically, the extra automorphisms are the classical Cremona symmetries of the degree $5$ del Pezzo surface, and modularly they allow one to change the distinguished label.
\end{Example}

\bibliographystyle{amsalpha}
\bibliography{Biblio}
\end{document}